\documentclass[amsfonts]{amsart}

\address{Heng Xie, Fachgruppe Mathematik and informatik, Bergische Universit\"{a}t Wuppertal,  Gau{\ss}stra{\ss}e 20, 42119 Wuppertal, Germany}
\email{heng.xie@math.uni-wuppertal.de}
\usepackage[margin=1.2in]{geometry}
\usepackage{amsmath}
\usepackage{amsfonts}
\usepackage{amsthm}
\usepackage{amssymb}
\usepackage{amscd}
\usepackage{mathrsfs}
\usepackage{graphicx}
\usepackage{rotating, graphicx}
\usepackage[all,cmtip]{xy}
\usepackage{enumerate}
\usepackage{tabularx}
\usepackage{tikz-cd}
\usepackage[thinlines]{easytable}
\usepackage{multirow}
\usepackage{hyperref}

\numberwithin{equation}{section}

\theoremstyle{plain}
\newtheorem{theo}{Theorem}[section]

\theoremstyle{theorem}
\newtheorem{lem}[theo]{Lemma}
\newtheorem{Prop}[theo]{\bf Proposition}
\newtheorem{Coro}[theo]{\bf Corollary}

\theoremstyle{definition}
\newtheorem{remark}[theo]{Remark}

\newtheorem{Def}[theo]{Definition}

\newcommand{\A}{\mathbb{A}}

\newcommand{\C}{\mathbb{C}}

\newcommand{\D}{\mathcal{D}}

\newcommand{\SO}{\mathcal{O}}

\newcommand{\Z}{\mathbb{Z}}

\newcommand{\PS}{\mathbb{P}}
\newcommand{\id}{\textnormal{id}}
\newcommand{\pt}{\textnormal{pt}}
\newcommand{\op}{\textnormal{op}}

\newcommand{\Hom}{\textnormal{Hom}}
\newcommand{\quis}{\textnormal{quis}}
\newcommand{\sPerf}{\textnormal{Ch}^b}

\newcommand{\hMod}{\textnormal{-Mod}}

\newcommand{\can}{\textnormal{can}}

\newcommand{\LC}{\mathcal{L}}
\newcommand{\GC}{\mathcal{G}}

\newcommand{\Spec}{\textnormal{Spec}}
\newcommand{\Ext}{\mathcal{E}\textnormal{xt}}
\newcommand{\ext}{\textnormal{Ext}}

\newcommand{\ckr}{\textnormal{coker}}

\newcommand{\Sm}{\textnormal{Sm}}

\newcommand{\sHom}{\mathcal{H}\textnormal{om}}
\newcommand{\M}{\mathcal{M}}
\newcommand{\Q}{\mathcal{Q}}
\newcommand{\Tot}{\textnormal{Tot}}
\newcommand{\VB}{\textnormal{VB}}
\newcommand{\Ch}{\textnormal{Ch}}
\newcommand{\I}{\mathcal{I}}

\newcolumntype{M}[1]{>{\centering\arraybackslash}m{#1}}
\newcolumntype{N}{@{}m{0pt}@{}}

\newcommand{\xdownarrow}[1]{%
	{\left\downarrow\vbox to #1{}\right.\kern-\nulldelimiterspace}
}
\newcommand{\xuparrow}[1]{%
	{\left \uparrow\vbox to #1{}\right.\kern-\nulldelimiterspace}
}
\makeatletter
\newcommand{\extp}{\@ifnextchar^\@extp{\@extp^{\,}}}
\def\@extp^#1{\mathop{\bigwedge\nolimits^{\!#1}}}
\makeatother

\makeatletter
\newcommand{\subjectclass}[2][ ]{%
	\let\@oldtitle\@title%
	\gdef\@title{\@oldtitle\footnotetext{#1 \emph{MSC:} #2}}%
}
\makeatother
\title{A Transfer morphism for Hermitian $K$-theory of schemes with involution}

\author{ Heng Xie}
\subjectclass{Primary: 19G38; Secondary: 19G12} 
\begin{document}

	\begin{abstract} In this paper, we consider the Hermitian $K$-theory of schemes with involution, for which we construct a transfer morphism and prove a version of the d\'{e}vissage theorem. This theorem is then used to compute the Hermitian $K$-theory of $\mathbb{P}^1$ with involution given by $[X:Y] \mapsto [Y:X]$. We also prove the $C_2$-equivariant $\A^1$-invariance of Hermitian $K$-theory, which confirms the representability of Hermitian $K$-theory in the $C_2$-equivariant motivic  homotopy category  of Heller, Krishna and \O stv\ae r \cite{HKO14}.  \end{abstract}
	
	\maketitle
	
	\section{Introduction}
	In the 1970s, researchers found that to understand quadratic forms, it was helpful to study Witt groups of function fields of algebraic varieties, which has led to substantial progress in quadratic form theory (see \cite{Scha85}, \cite{Lam05}, \cite{KS80} and \cite{Knu91}).\ However, for many function fields, Witt groups are very difficult to understand. Instead of computing the Witt groups of function fields of algebraic varieties, Knebusch \cite{Kne77} proposed to study the Witt groups of the algebraic varieties themselves, which would certainly reveal some information about their function fields counterparts. In the introduction of \textit{loc. cit.}, Knebusch suggested developing a version of Witt groups of schemes with involution, which would not only enlarge the theory of trivial involution but also provides new insights into Topology. For instance, the $L$-theory of the Laurent polynomial ring $k[T,T^{-1}]$ with the non-trivial involution $T \mapsto T^{-1}$ turned out to be very useful in understanding geometric manifolds \cite{Ran98}, and Witt groups can be identified with  $L$-theory if two is invertible.  
	
	More generally, Witt groups fit into the framework of Hermitian $K$-theory, \cite{Ba73}, \cite{Kar80} and \cite{Sch17}.  More precisely, the negative homotopy groups of the Hermitian $K$-theory spectrum are Witt groups, cf. \cite[Proposition 6.3]{Sch17}. It is also worth mentioning that Hermitian $K$-theory has been successfully applied to solve several problems in the classification of vector bundles and the theory of Euler classes (cf. \cite{AF14a}, \cite{AF14b} and \cite{FS09}). In light of this, I decided to develop the current paper within the framework of the Hermitian $K$-theory of dg categories of Schlichting \cite{Sch17}. 
	
	The simplest example of Witt groups of schemes with non-trivial involution is     $W(\Spec(\mathbb{C}),\sigma)$  where $\sigma$ is complex conjugation (here we consider $\Spec (\mathbb{C})$ over $\Spec(\Z[\frac{1}{2}])$).\ We have $W(\Spec(\mathbb{C}),\sigma) \cong \mathbb{Z}$ by taking the rank of positive definite diagonal Hermitian forms over $\C$ (cf. \cite[I.10.5]{Knu91}). This computation  is different from $W(\Spec(\mathbb{C})) \cong \mathbb{Z}/2\mathbb{Z}$ in which case every rank two quadratic form is hyperbolic.

	On the one hand, it is known that the topological Hermitian $K$-theory of spaces with involution is equivalent to Aityah's $KR$-theory (\cite{At66}).\ On the other hand, the Hermitian $K$-theory of schemes with involution is a natural lifting of the topological Hermitian $K$-theory to the algebraic world. In light of this, the Hermitian $K$-theory of schemes with involution could be regarded as a version of  $KR$-theory in algebraic geometry. By working with Hermitian $K$-theory of schemes with involution, Hu, Kriz and Ormsby \cite{HKO11} managed to prove the homotopy limit problem for Hermitian $K$-theory.
	
	In this paper, we generalize Gille's transfer morphism on Witt groups \cite{Gil03} and Schlichting's transfer morphism on Grothendieck-Witt groups \cite[Theorem 9.19]{Sch17} to non-trivial involution. More precisely, in Theorem \ref{thm:transfer} we prove the following:
	\begin{theo}
		Let $(X,\sigma_X)$ and $(Z, \sigma_Z)$ be schemes with involution and with $\frac{1}{2}$ in their global sections.  Suppose that $(X,\sigma_X)$ has a dualizing complex with involution $(I,\sigma_I)$ (cf. Definition \ref{Dwin}).\ If $ \pi: Z \rightarrow X$ is  a finite morphism of schemes with involution, then the direct image functor $\pi_*: \Ch^b_c(\Q(Z)) \rightarrow \Ch^b_c(\Q(X))  $ induces a map of spectra
		\[T_{Z/X} : GW^{[i]} (Z,\sigma_Z, (\pi^{\flat} I^\bullet, \sigma_{\pi^{\flat}I}) ) \rightarrow GW^{[i]}(X,\sigma_X, (I^\bullet, \sigma_I) ). \] 
	\end{theo}
	In light of the work of Balmer-Walter \cite{BW02} and Gille \cite{Gil07a}, we also prove a version of the d\'{e}vissage theorem for Hermitian $K$-theory of schemes with involution. The following result is proved in Theorem \ref{thm:devissage}.
	\begin{theo}\label{thm:devissage}
		Let $(X,\sigma_X)$ and $(Z, \sigma_Z)$ be schemes with involution and  with $\frac{1}{2}$ in their global sections. Suppose that  $(X,\sigma_X)$ has a minimal dualizing complex with involution $(I,\sigma_I)$  (cf. Definition \ref{Dwin}).\ If $ \pi: Z \hookrightarrow X$ is a closed immersion which is invariant under involutions, then the direct image functor $\pi_*: \Ch^b_c(\Q(Z)) \rightarrow \Ch^b_{c,Z}(\Q(X))  $ induces an equivalence of spectra
		\[ D_{Z/X} : GW^{[i]} (Z,\sigma_Z, (\pi^{\flat} I^\bullet, \sigma_{\pi^{\flat}I}) ) \rightarrow GW^{[i]}_Z(X,\sigma_X, (I^\bullet, \sigma_I) ). \] 		
	\end{theo}
	
	Note that a Gersten-Witt complex for Witt groups of schemes with involution was constructed by Gille \cite{Gil09}, but the above formula seems not to appear anywhere in the literature. If $X$ and $Z$ are both regular, we have the following more precise formulation of the d\'{e}vissage theorem (cf. Theorem \ref{thm:Devissage}).
	\begin{theo}[D\'evissage]Let $(X,\sigma_X)$ be a regular scheme with involution.  Let $(\LC,\sigma_\LC)$ be a dualizing coefficient with $\LC$ a locally free $\SO_X$-module of rank one. If $Z$ is a regular scheme regularly embedded in $X$ of codimension $d$ which is invariant under $\sigma$, then there is an equivalence of spectra
		\[ D_{Z/X, \mathcal{L}} : GW^{[i-d]} (Z,\sigma_Z, (\omega_{Z/X} \otimes_{\SO_{X}} \mathcal{L}, \sigma_{\omega_{Z/X}} \otimes  \sigma_{\mathcal{L} }) )  \longrightarrow GW^{[i]}_Z(X,\sigma_X, (\mathcal{L}, \sigma_{\mathcal{L}}) ). \] 
	\end{theo}

	It turns out that if the involution is non-trivial, the canonical double dual identification has a certain sign varying according to the data of dualizing coefficients (cf.\ Definition \ref{D}). 
	This sign is not important if the involution is trivial, but it is crucial if the involution is non-trivial. As an application, we use the d\'{e}vissage theorem to obtain the following result (cf. Theorem \ref{P1switch}). 
	\begin{theo} Let $S$ be a regular scheme with $\frac{1}{2} \in \SO_X$. Then, we have an equivalence of spectra
		\[ GW^{[i]} (\PS^1_S, \tau_{\PS^1_S}) \cong GW^{[i]} (S) \oplus GW^{[i+1]}(S) \]
		where $\tau_{\PS^1_S}  : \PS^1_S \rightarrow \PS^1_S : [x: y] \mapsto [y:x]$.\ In particular, we have the following isomorphism on Witt groups
		\[W^{i}(\PS^1_S, \tau_{\PS^1_S}) \cong W^{i} (S) \oplus W^{i+1}(S). \]
	\end{theo}
	It is well-known that $W^{i}(\PS^1_S) \cong W^{i} (S) \oplus W^{i-1}(S)$ if the involution on $\PS^1_S$ is trivial (cf. \cite{Ne09} and \cite{Wal03}). By our d\'{e}vissage theorem, the involution $\sigma: \PS^1_S \rightarrow \PS^1_S : [x: y] \mapsto [y:x]$ induces a sign $(-1)$ on the double dual identification, and we have 
	\[W^{i}(\PS^1_S, \sigma) \cong W^{i} (S) \oplus W^{i-1}(S, - \can)\]
	but $W^{i-1}(S, - \can) \cong W^{i+1}(S)$. 
	
	
	We also use the d\'{e}vissage to prove the following $C_2$-equivariant $\A^1$-invariance (cf. Theorem \ref{thm:c2a1invariance}). 
	
	\begin{theo}\label{c2a1i} Let $S$ be a regular scheme with involution $\sigma_S$ and with $\frac{1}{2} \in \SO_S$. Let $\sigma_{\A^1_S}: \A^1_S \rightarrow \A^1_S$ be the involution on $\A^1_S$ with the indeterminant fixed by $\sigma_{\A^1_S}$ and such that the following diagram commutes.
		\[ \begin{CD}
		\A^1_S @>\sigma_{\A^1_S}>> \A^1_S\\
		@VpVV @VpVV \\
		S        @>\sigma_{S}>> S
		\end{CD}\] 
		The pullback 
		\[ p^*:   GW^{[i]} (S,\sigma_S) \rightarrow GW^{[i]} (\A^1_S , \sigma_{\A^1_S})  \]
		is an isomorphism.
	\end{theo}
	For affine schemes with involution Theorem \ref{c2a1i} was proved by Karoubi  (cf. \cite[Part II]{Kar74}), but it was not known for regular schemes with involution. If $S$ is quasi-projective, then it can be covered by invariant affine open subschemes, and we can use the Mayer-Vietoris of Schlichting (\cite{Sch17}) to reduce the quasi-projective case to the affine case. However, the method of Meyer-Vietoris does not work if $S$ is not quasi-projective because there exist non-quasi-projective regular schemes which can not be covered by invariant affine opens. 
	
	It is easy to see that schemes with involution can be identified with schemes with $C_2$-action where $C_2$ is the cyclic group of order two considered as an algebraic group. Together with a variant of Nisnevich excision \cite[Theorem 9.6]{Sch17} adapted to schemes with involution, we obtain the $C_2$-representability of Hermitian $K$-theory in the $C_2$-equivariant motivic homotopy category $\mathcal{H}^{C_2}_\bullet(S)$ of Heller, Krishna and \O stv\ae r \cite{HKO14}.\  The following result is proved in Theorem  \ref{thm:C2representablity}. 
	\begin{theo} Let $S$ be a regular scheme with $\frac{1}{2} \in \SO_S$, and let $(X,\sigma) \in \textnormal{Sm}^{C_2}_S$. Then, there is a bijection of sets
		\[[S^n \wedge (X,\sigma)_+, GW^{[i]}]_{\mathcal{H}^{C_2}_\bullet(S)} = GW_n^{[i]}(X,\sigma). \]
	\end{theo}

	
	\section{Notations and prerequisites}\label{sec:notation}
	In this paper, we assume that every scheme is Noetherian and has $\frac{1}{2}$ in its global sections, and every ring is commutative Noetherian with identity and $\frac{1}{2}$. If $\mathcal{E}$ is an exact category, we write $\D^b(\mathcal{E})$ the bounded derived category of $\mathcal{E}$. 
	
	\subsection{Notations on rings} Let $R$ be a ring. Throughout the paper, we will make use of the following notations associated with $R$.
	\begin{itemize}
		\item[-] $ R\hMod$ is the category of (left) $R$-modules.
		\item[-] $\M(R)$ is the category of finitely generated $R$-modules.
		\item[-] $\M_{fl}(R)$ is the category of finite length $R$-modules.
		\item[-] $\M_J(R)$ is the category of finitely generated $R$-modules supported in an ideal $J$ of $R$.\ Recall that $\text{Supp}(M) = \{ \mathfrak{p} \in \Spec(R) | M_{\mathfrak{p}} \neq 0 \}$ and $V(J) = \{\mathfrak{p} \in \Spec(R) | \mathfrak{p} \supset J \}$. More precisely, $\M_J(R)$ is the full subcategory of $\M(R)$ consisting of those modules $M$ such that $\text{Supp}(M) \subset V(J)$. 
		\item[-] $\D^b_c(R\hMod)$ is the derived category of bounded complexes of  $R$-modules with coherent cohomology. 
		\item[-] $\D^b_{fl}(R)$ is the full subcategory of $\D^b_c(R\hMod)$ of complexes, whose cohomology modules are finite length $R$-modules. 
		
		\item[-]  $\D^b_J(R)$ is the full triangulated subcategory of $\D^b_c(R\hMod)$ of complexes whose cohomology modules are annihilated by some power of the ideal $J$.\ 
	\end{itemize}

	\begin{Def} An \textit{involution} $\sigma$ on a ring $R$ is a ring homomorphism $\sigma: R\rightarrow R$ such that $\sigma^2 = \id_R$. 
	\end{Def} 

	 Let $(R,\sigma)$ be a ring with involution and $M$ a left $R$-module. Define $M^{\op_R}$ (or simply $M^\op$) to be the same as $M$ as a set. For an element in $M^{\op}$, we use the symbol $m^{\op}$ to denote the element $m$ coming from the set structure of $M$.  Consider $M^\op$ as a right $R$-module via $m^\op a = (\sigma(a)m)^\op$.

	\begin{Def}\label{Def:dualitycoefficient} \cite[Section 7.4]{Sch10b}
		Let $(R,\sigma)$ be a ring with involution. A \textit{duality coefficient} $(I,i)$ on the category $R\hMod$ with respect to the involution $\sigma$ consists of an $R$-module $I$ equipped with an $R$-module isomorphism $i:I \rightarrow I^\op$ such that $i^2 = \id$.
	\end{Def}

 Let $\sigma:R \rightarrow R'$ be a ring homomorphism. Let $M'$ be a module over $R'$. We can form 
	the restriction of scalars 
	\[\sigma_* M' : = M'|_{R}.\] 
	which can also be considered as an $R$-module.

	
			\begin{remark}
		Note that if $(R,\sigma)$ is a ring with involution and $M$ is an $R$-module, then $\sigma_* M = M^\op$. 
	\end{remark}

	\begin{remark}\label{rmk:affineiden}
		For a duality coefficient $(I,i)$, one can form a category with duality $(R\hMod, \#^I_\sigma, \can^{I}_\sigma)$ where
		 \[ \#^I_\sigma: (R \hMod)^\op \rightarrow R\hMod: M \mapsto \Hom_R(\sigma_*M, I)   \]
		is a functor and  \[\can^I_{\sigma,M}: M \rightarrow M^{ \#^I_\sigma\#^I_\sigma } : \can_{I,M}^\sigma(x)(f^\op) = i(f(x^\op)) \]  
		is a natural morphism of $R$-modules (cf.  \cite[Section 7.4]{Sch10b}).
	\end{remark}

 The following lemma will be used in Theorem \ref{WDevissage}.
 
	\begin{lem}\label{lem: RSf}
	 Let $\sigma:R\rightarrow R'$ and $\tau: S \rightarrow S'$  isomorphisms of rings, and suppose that $\sigma' : = \sigma^{-1}$ and $ \tau': = \tau^{-1}$. Suppose furthermore that one has the following commutative diagram of rings.
	\[\xymatrix{ R \ar[d]_-{g} \ar[r]^-{\sigma} & R' \ar[d]_-{g'} \ar[r]^-{\sigma'} & R \ar[d]_-{g} \\
		S \ar[r]^-{\tau} & S' \ar[r]^-{\tau'} & S	
	}\] 
	 Let $I$ be an $R$-module and let $I'$ be an $R'$-module.  Assume that we have an $R$-module isomorphism $\sigma_I :I \rightarrow \sigma_*I'$.
	
	\noindent \textnormal{(1)}.  Let $M$ be an $R$-module and let $M'$ be an $R'$-module.\ Assume that we have an $R$-module isomorphism $\sigma_{M'} :M' \rightarrow \sigma'_*M$. Then, the map
	\[\epsilon_{M,I}:  \Hom_R(M,I) \rightarrow \sigma_*\Hom_{R'}(M',I'): f \mapsto \sigma_I f \sigma_{M'}\]
	is an $R$-module isomorphism.
	
	\noindent \textnormal{(2)}. Let $N$ be an $S$-module and let $N'$ be an $S'$-module.  Assume that we have an $S$-module isomorphism $\sigma_{N'}:N' \rightarrow \tau'_*N$. The set $\Hom_R(N,I)$ has an $S$-module structure and the map 
	\[\epsilon_{N,I}:  \Hom_R(N,I) \rightarrow \tau_*\Hom_{R'}(N',I'): f \mapsto \sigma_I f \sigma_{N'}\]
	is a well-defined $S$-module isomorphism. 

\end{lem}

	\subsection{Notations on schemes}
	Let $X$ be a scheme. We fix the following notations that will be needed in the remaining sections.
	\begin{itemize}
		\item[-] $ \SO_X\hMod$ is the category of $\SO_X$-modules.
		\item[-] $\Q(X)$ is the category of quasi-coherent $\SO_X$-modules.
		\item[-] $\VB(X)$ is the category of finite rank locally free coherent modules over $X$.
		\item[-] $\M_Z(X)$ is the category of finitely generated $\SO_X$-modules supported in a closed scheme $Z$ of $X$.\ Recall that $\text{Supp}(\mathcal{F}) = \{ x \in X | \mathcal{F}_x \neq 0 \}$. More precisely, $\M_Z(X)$ is the full subcategory of $\M(X)$ consisting of those modules $\mathcal{F}$ such that $\text{Supp}(\mathcal{F}) \subset Z$. 
			\item[-] $K^b_c(\Q(X))$ is the homotopy category of bounded complexes of quasi-coherent sheaves on $X$ with coherent cohomology. 
		\item[-] $\D^b_c(\Q(X))$ is the derived category of bounded complexes of quasi-coherent sheaves on $X$ with coherent cohomology. 
		\item[-] $\D^b_{c,Z}(\Q(X))$ is the full triangulated subcategory of $\D^b_c(\Q(X))$ consisting of complexes $\mathcal{F}^\bullet$ such that 
		$\{ x \in X | \mathcal{H}^i(\mathcal{F}^\bullet)_x \neq 0 \textnormal{ for some $i \in \Z$ }  \} \subseteq Z$. 
		\item[-] $\D^b(X)$ is the derived category $\D^b(\VB(X))$.
		\item[-] $\D^b_Z (X)$ is the full triangulated category of  $\D^b (X)$ whose objects are the complexes $\mathcal{F}^\bullet$ such that 
		$\{ x \in X | \mathcal{H}^i(\mathcal{F}^\bullet)_x \neq 0 \textnormal{ for some $i \in \Z$ }  \} \subseteq Z$. 
		\item[-] $\Ch^b_c(\Q(X))$ is the dg category of bounded complexes of quasi-coherent $\SO_X$-modules with coherent cohomology. 
		\item[-] $\Ch^b_{c,Z}(\Q(X))$ is the  full dg subcategory of  $\Ch^b_c(\Q(X))$ consisting of complexes $\mathcal{F}^\bullet$ such that 
		$\{ x \in X | \mathcal{H}^i(\mathcal{F}^\bullet)_x \neq 0 \textnormal{ for some $i \in \Z$ }  \} \subseteq Z$. 
		\item[-] $\Ch^b (X)$ is the dg category $\Ch^b(\VB(X))$ of bounded complexes of locally free coherent $\SO_X$-modules. 
		\item[-] $\Ch^b_Z (X)$ is the full dg subcategory $\Ch^b(X)$ of those complexes $\mathcal{F}^\bullet$ such that 
		$\{ x \in X | \mathcal{H}^i(\mathcal{F}^\bullet)_x \neq 0 \textnormal{ for some $i \in \Z$ }  \} \subseteq Z$. 
	\end{itemize}

	\begin{Def} An \textit{involution} $\sigma_X$ on a ringed space $X$ is a morphism $\sigma_ X: X\rightarrow X$ of ringed spaces such that $\sigma^2 = \id_X$. 
	\end{Def} 
	\begin{remark}
		If $X = \Spec(R)$, to give a scheme with involution $(X, \sigma)$ is the same as giving a ring with involution $(R,\sigma)$.  
	\end{remark}

	\begin{Def}\label{D}
		A \textit{duality coefficient} $(I,i)$ for the category of $\SO_{X}$-modules on $X$ with respect to the involution $\sigma_X$ consists of an $\SO_{X}$-module $I$ equipped with an $\SO_X$-module isomorphism $i:I \rightarrow \sigma_* I$ such that $\sigma_* i \circ i = \id$.  An isomorphism $(I,i) \rightarrow (I'.i')$ of duality coefficients is an isomorphism of $\SO_{X}$-modules $\alpha: I \rightarrow I'$ such that the following diagram commutes
		\[ \begin{CD}
		I @>\alpha>> I'\\
		@ViVV @V{i'}VV \\
		\sigma_*I @>{\sigma_*\alpha}>> \sigma_*I'.
		\end{CD}\]
	\end{Def}
	
	\begin{remark}
Let $(X, \sigma)$ be a scheme with involution. 
By globalizing the affine case in Remark \ref{rmk:affineiden}, we  obtain a category with duality 
\[(\SO_X\hMod, \#_\sigma^I, \can_\sigma^I).\]
If the double dual identification $\can_\sigma^I$ is clear from the context, we will use the simplified notation $(\SO_X\hMod, \#_\sigma^I)$ instead of the more precise $(\SO_X \hMod, \#_\sigma^I, \can_\sigma^I)$.
	\end{remark} 
	\begin{lem}\label{II'}
		Let $\alpha:(I,i)\rightarrow (I',i')$ be an isomorphism of duality  coefficients. Then, the identity functor induces a duality preserving equivalence of categories with duality
		\[(\id, \alpha_*):(\SO_X \hMod, \#_\sigma^I) \rightarrow (\SO_X \hMod, \#_{\sigma}^{I'}) \] \end{lem}
	\begin{proof}
		The duality compatibility isomorphism is defined by 
		\[      \alpha_*:\sHom_{\SO_X}(\sigma_*M, I) \rightarrow \sHom_{\SO_X}(\sigma_*M,I')   :f \mapsto \alpha\circ f     \] 
		which is natural in $M$. Now it is straightforward to check the axioms. 
	\end{proof}

	\subsection{Dualizing complexes with involution} Recall the following definition of a duality complex from \cite[Definition 1.7]{Gil07b} or \cite{Ha66}.  
	
	\begin{Def}
		Let $X$ be a scheme.\ A \textit{dualizing complex} $I^\bullet$ on $X$ is a complex  of injective modules
		\[I^\bullet :=  (\cdots  \rightarrow 0 \rightarrow I^m \stackrel{d^m}\longrightarrow I^{m+1}\rightarrow \cdots \rightarrow I^{n-1} \stackrel{d^{n-1}}\longrightarrow I^{n} \rightarrow 0 \rightarrow \cdots )  \in \D^b_{c}(\Q(X))\] such that the natural morphism of complexes
		\[\can_{I,M} : M^\bullet \rightarrow \sHom_{\SO_{X}}(\sHom_{\SO_{X}}(M^\bullet,I^\bullet), I^\bullet ) : m \mapsto \big( \can_{I,M}(m) : f \mapsto (-1)^{|f||m|} f(m) \big)\]
		is a quasi-isomorphism for any $M^\bullet \in \D^b_{c}(\Q(X))$. A dualizing complex is called \textit{minimal} if $I^r$ is an essential extension of $\ker(d^r)$ for all $r \in \Z$. 
	\end{Def}
	
	\begin{remark}
		Any dualizing complex is quasi-isomorphic to a minimal dualizing complex \cite[Proposition 1.9]{Gil07b}. 
	\end{remark}
	
	\begin{Def}\label{Dwin}
		Let $(X,\sigma)$ be a Noetherian scheme with involution and $\frac{1}{2} \in \SO_X$. A \textit{dualizing complex with involution} on $(X,\sigma)$ is a pair $(I^\bullet,  \sigma_{I})$ consisting of a dualizing complex $ I^\bullet$ and an isomorphism of complexes of $\SO_X$-modules $\sigma_{I}: I^\bullet \rightarrow \sigma_*(I^\bullet)$ such that $\sigma_*(\sigma_I) \circ \sigma_I = \id_I$. 
	\end{Def}
	\begin{remark}
		Note that $(I^p,\sigma_{I^p})$ is a duality coefficient for the category $\SO_X \hMod$. 
	\end{remark}
	\begin{lem}\label{lem:goe-dualizingcomplex} Let $(X,\sigma)$ be a Gorenstein scheme of finite Krull dimension with involution and $\frac{1}{2} \in \SO_X$. Let $(\LC,\sigma_\LC)$ be a dualizing coefficient with $\LC$ a locally free $\SO_X$-module of rank one. Then, there exists a minimal dualizing complex with involution $(J^\bullet,\sigma_J)$, which is quasi-isomorphic to $(\LC,\sigma_\LC)$. 
	\end{lem}
	\begin{proof}
		Let $X^{(p)} := \{ x \in X| \dim \SO_{X,x} = p \}$.\ Recall  \cite[p241]{Ha66} that we have a Cousin complex
		\[  0 \rightarrow \SO_X \stackrel{d^{-1}}\longrightarrow I^0 \stackrel{d^{0}}\longrightarrow I^1 \longrightarrow \cdots \longrightarrow I^n \longrightarrow 0 \longrightarrow \cdots \]
		with $I^0 := \bigoplus_{x \in X^{(0)}} i_{x*} (\SO_{X,x})$ and $I^p := \bigoplus_{x\in X^{(p)}} i_{x*}(\ckr (d^{t-2})_{x})$ where the map $i_x: \Spec(\SO_{X,x}) \rightarrow X$ is the canonical map. If $X$ is Gorenstein, the Cousin complex provides a minimal dualizing complex $I^\bullet$ on $X$ quasi-isomorphic to $\SO_X$. This can be checked locally on closed points by \cite[Corollary V.2.3 p259]{Ha66}. For the case of Gorenstein local rings, see \cite[Theorem 5.4]{Sha69}.  
		
		We define a map $\sigma_I: I^p \rightarrow \sigma_* I^p$ as follows. 
		We consider the following commutative  diagram
		\[\xymatrix{ \Spec(\SO_{X,x}) \ar[d]^-{i_x} \ar[r]^-{\sigma_{X,x}} & \Spec(\SO_{X,\sigma(x)}) \ar[d]^-{i_{\sigma(x)}} \\
			X \ar[r]^-{\sigma_X} & X}\]
		Assume that $M$ is an $\SO_{X}$-module together with a map $M \rightarrow \sigma_{*}M$. If $x = \sigma(x)$, the map $M \rightarrow \sigma_{*}M$ induces a canonical map $ i_{x,*} i_x^* M \rightarrow \sigma_* i_{x,*} i_x^* M  $ of $\SO_X$-modules by using the natural isomorphism $\sigma_* i_{x,*} i_x^* M \cong i_{\sigma(x),*} \sigma_{x,*}  i_{x}^* M  \cong  i_{\sigma(x),*} i_{\sigma(x)}^* \sigma_* M $. If $x \neq \sigma(x)$, the map $M \rightarrow \sigma_{*}M$ still induces a canonical map \[ i_{x,*} i_x^* M \oplus i_{\sigma(x),*} i_{\sigma(x)}^* M \rightarrow \sigma_* i_{\sigma(x),*} i_{\sigma(x)}^* M   \oplus\sigma_* i_{x,*} i_x^* M \]
		by the functorial isomorphism above. The map $\sigma_I: I^p \rightarrow \sigma_* I^p$ is obtained by applying this construction inductively. (Note that $x \in X^{(p)}$ if and only if $\sigma(x) \in X^{(p)}$).  By our construction it is clear that $\sigma_*(\sigma_I) \circ \sigma_I = \id_I$.
		
		If $\LC$ is a line bundle on $X$, then $J^\bullet:= \LC \otimes I^\bullet$ is  a dualizing complex, cf.\ \cite[Proof (1) of Theorem V.3.1, p. 266]{Ha66}. Moreover, the complex $J^\bullet$ is quasi-isomorphic to $\LC$. Note that we have a canonical isomorphism $\sigma_*(\LC \otimes I^\bullet) \rightarrow \sigma_* \LC \otimes \sigma_* I^\bullet $. These observations reveal a canonical involution $ \sigma_J: J^\bullet \rightarrow \sigma_* J^\bullet$.
	\end{proof}

	\subsection{Setup for Hermitian $K$-theory of schemes with involution } Let $(X,\sigma)$ be a scheme with involution. In this paper, we work within the framework of Schlichting \cite{Sch17}. Recall that $\Ch^b_c(\Q(X))$ is a closed symmetric monoidal category under the tensor product of complexes $M^\bullet \otimes_{\SO_{X}} N^\bullet$ given by
	\[  (M^\bullet \otimes_{\SO_{X}} N^\bullet)^n : = \bigoplus_{i+j = n} M^i \otimes_{\SO_{X}} N^j, ~~~ d (m \otimes n ) = dm \otimes n + (-1)^{|m||n|} m \otimes dn  \]
	and internal homomorphism complexes $ \sHom_{\SO_{X}}(M^\bullet, N^\bullet)$ given by
	\[ \sHom_{\SO_X}(M^\bullet, N^\bullet)^n := \prod_{j-i = n} \sHom_{\SO_X}(M^i, N^j), ~~~ df = d \circ f -(-1)^{|f|} f \circ d. \]
	Let $ (I^\bullet,\sigma_{I})$ be a dualizing complex with involution.  Then, we define the duality functor 
	\[ \#^{I^\bullet}_\sigma: (\Ch^b_c(\Q(X)))^\op \rightarrow \Ch^b_c(\Q(X)) : E^\bullet \mapsto \sHom_{\SO_{X}}(\sigma_*E^\bullet, I^\bullet ) ,\] and  the double dual identification $ \can^{I^\bullet}_{\sigma, E} : E \rightarrow E^{ \#^{I^\bullet}_\sigma  \#^{I^\bullet}_\sigma}  $ given by the composition
	\[ E \stackrel{\can_E}\longrightarrow \sHom_{\SO_{X}}( \sHom_{\SO_{X}}(E ,I), I) \longrightarrow \sHom_{\SO_{X}}(\sHom_{\SO_{X}}(E, \sigma_*I), I )  \longrightarrow  \sHom_{\SO_{X}}(\sigma_* \sHom_{\SO_{X}}(\sigma_*  E, I), I ).  \] 
	Here the first map $\can_E$ is defined by $\can_E (x)(f) = (-1)^{|x||f|} f(x)$ as in the case of trivial duality, the second is induced by $\sigma_I: I \rightarrow\sigma_*I$, and the third is induced by $\sigma_{*} \sHom_{\SO_{X}}(M, N) =  \sHom_{\SO_{X}}(\sigma_{*}M,\sigma_{*} N) $.
	\begin{lem}
		The quadruple 
		\[ (\Ch^b_c(\Q(X)), \quis,  \#^{I^\bullet}_\sigma , \can^{I^\bullet}_\sigma )\]
		is a dg category with weak equivalence and duality.
	\end{lem} 
	\begin{proof}
		It is enough to check that\[ (\Ch^b_c(\Q(X)),   \#^{I^\bullet}_\sigma , \can^{I^\bullet}_\sigma )\]
		is a category with duality, see \cite[Section 3.7]{Gil09}. 
	\end{proof}
	\begin{Def}
		We define the \textit{coherent Hermitian $K$-theory spectrum of schemes with involution} of $(X,\sigma_X)$ with respect to $(I,\sigma_I)$ as 
		\[GW^{[i]}(X,\sigma_X, (I, \sigma_I)) := GW^{[i]}(\Ch^b_c(\Q(X)), \quis,  \#^{I^\bullet}_\sigma , \can^{I^\bullet}_\sigma )\]
		for every $i \in \Z$, where $GW^{[i]}(\Ch^b_c(\Q(X)), \quis,  \#^{I^\bullet}_\sigma , \can^{I^\bullet}_\sigma )$ is the Grothendieck-Witt spectrum of the dg category with weak equivalences and duality $(\Ch^b_c(\Q(X)), \quis,  \#^{I^\bullet}_\sigma , \can^{I^\bullet}_\sigma )$. The \textit{Witt group} of $(X,\sigma_X)$ with respect to $(I,\sigma_I)$ is defined as
		\[W^{i}(X,\sigma_X, (I, \sigma_I)) := W^{i}(\D^b_c (\Q(X)),   \#^{I^\bullet}_\sigma , \can^{I^\bullet}_\sigma ), \]
		where $W^{i}(\D^b_c (\Q(X)),   \#^{I^\bullet}_\sigma , \can^{I^\bullet}_\sigma )$ is the Witt group of the triangulated category with duality $(\D^b_c (\Q(X)),   \#^{I^\bullet}_\sigma , \can^{I^\bullet}_\sigma )$. Let $Z$ be an invariant closed subscheme of $X$. We can also define 
		\[ GW^{[i]}_Z(X,\sigma_X, (I, \sigma_I)) := GW^{[i]}(\Ch^b_{c,Z}(\Q(X)), \quis,  \#^{I^\bullet}_\sigma , \can^{I^\bullet}_\sigma ) \]
		and
		\[ W^{i}_Z(X,\sigma_X, (I, \sigma_I)) := W^{i}(\D^b_Z (X),   \#^{I^\bullet}_\sigma , \can^{I^\bullet}_\sigma ). \]
	\end{Def}
	Let $(\LC,\sigma_\LC)$ be a dualizing coefficient with $\LC$ a locally free $\SO_X$-module of rank one. Then, we have a dg category with duality
	\[ (\Ch^b(X),  \quis,  \#^{\LC}_\sigma , \can^{\LC}_\sigma ),\]
	where 
	\[ \#^{\LC}_\sigma: (\Ch^b (X))^\op \rightarrow \Ch^b (X) : E^\bullet \mapsto \sHom_{\SO_{X}}(\sigma_*E^\bullet, \LC ) ,\] and  the double dual identification $ \can^{\LC}_{\sigma, E} : E \rightarrow E^{ \#^{\LC}_\sigma  \#^{\LC}_\sigma}$ is given by the composition
	\[ E \stackrel{\can_E}\longrightarrow \sHom_{\SO_{X}}( \sHom_{\SO_{X}}(E ,\LC), \LC) \longrightarrow \sHom_{\SO_{X}}(\sHom_{\SO_{X}}(E, \sigma_*\LC), \LC )  \longrightarrow  \sHom_{\SO_{X}}(\sigma_* \sHom_{\SO_{X}}(\sigma_*  E, \LC), \LC ).  \] 
	Here the first map $\can_E$ is defined by $\can_E (x)(f) = (-1)^{|x||f|} f(x)$, the second map is induced by $\sigma_\LC: \LC \rightarrow\sigma_*\LC$, and the third map is induced by $\sigma_{*} \sHom_{\SO_{X}}(M, N) =  \sHom_{\SO_{X}}(\sigma_{*}M,\sigma_{*} N) $.
	\begin{Def}
		We define the \textit{Hermitian $K$-theory spectrum of schemes with involution} of $(X,\sigma_X)$ with respect to $(\LC,\sigma_\LC)$ as 
		\[GW^{[i]}(X,\sigma_X, (\LC,\sigma_\LC)) := GW^{[i]}(\Ch^b (X), \quis,  \#^{\LC}_\sigma , \can^{\LC}_\sigma )\]
		for every $i \in \Z$. The \textit{Witt group} of $(X,\sigma_X)$ with respect to $(\LC,\sigma_\LC)$ is defined as
		\[W^{i}(X,\sigma_X,(\LC,\sigma_\LC)) := W^{i}(\D^b (X),   \#^{\LC}_\sigma , \can^{\LC}_\sigma ). \]
		Let $Z$ be an invariant closed subscheme of $X$. We can also define 
		\[ GW^{[i]}_Z(X,\sigma_X, (\LC,\sigma_\LC)) := GW^{[i]}(\Ch^b_Z (X), \quis,  \#^{\LC}_\sigma , \can^{\LC}_\sigma ) \]
		and
		\[ W^{i}_Z(X,\sigma_X, (\LC,\sigma_\LC)) := W^{i}(\D^b_Z(X),   \#^{\LC}_\sigma , \can^{\LC}_\sigma ). \]
	\end{Def}
	\subsection{Dualizing complex via a finite morphism} Let $(Z,\sigma_Z)$ be a scheme with involution, and let $\pi:Z \rightarrow X$ be a finite morphism. Assume that $\pi$ is compatible with the involutions on $Z$ and $X$, i.e., the following diagram commutes
	\[ \xymatrix{ Z \ar[d]_-{\sigma_Z}\ar[r]^-{\pi} & X \ar[d]_-{\sigma_X} \\
		Z \ar[r]^-{\pi} & X.} \] 
	
	Following Hartshorne \cite{Ha66}, we write
	\[\pi^\flat I^\bullet := \bar{\pi}^*  \sHom_{\SO_{X}}(\pi_*\SO_Z , I^\bullet  ),\]
	where $\bar{\pi} : (Z,\sigma_Z) \rightarrow (X,\pi_*\SO_Z) =: \bar{X} $ is the canonical flat map of ringed spaces. If $I^\bullet$ is  a dualizing complex, then $\pi^\flat I^\bullet$ is also a dualizing complex \cite[Section 2.4]{Gil07a}. Note that $\sigma_X: X\rightarrow X$ induces a canonical involution $\sigma_{\bar{X}}: \bar{X} \rightarrow \bar{X}$ defined on the topological spaces, as well as $\pi_* \sigma_Z: \pi_*\SO_Z \rightarrow \sigma_{X*} \pi_* \SO_Z = \pi_* \sigma_{Z*} \SO_Z $ on the sheaves of rings. Then, we can define an involution $\sigma_{\pi^\flat I^\bullet} : \pi^\flat I^\bullet \rightarrow \sigma_{Z*} \pi^\flat I^\bullet$ on $\pi^\flat I^\bullet$ through the following composition
	\[  \arraycolsep=1.7pt\def\arraystretch{1.7}
	\begin{array}{lclcllllllll}
	\pi^\flat I^\bullet = \bar{\pi}^*  \sHom_{\SO_{X}}(\pi_*\SO_Z , I^\bullet  )  &\stackrel{[\sigma_Z,1]}\longrightarrow&  \bar{\pi}^*  \sHom_{\SO_{X}}(\pi_* \sigma_{Z*} \SO_Z , I^\bullet  )  &\longrightarrow&  \bar{\pi}^*   \sHom_{\SO_{X}}( \sigma_{X*} \pi_* \SO_Z , I^\bullet  ) &\stackrel{[1, \sigma_{I}]}\longrightarrow& \\
	\bar{\pi}^*   \sHom_{\SO_{X}}( \sigma_{X*} \pi_* \SO_Z , \sigma_{X*} I^\bullet  )  &\longrightarrow&\bar{\pi}^*  \sigma_{X*}   \sHom_{\SO_{X}}( \pi_* \SO_Z , I^\bullet  )  &=&   \bar{\pi}^*  \sigma_{\bar{X}*}   \sHom_{\SO_{X}}( \pi_* \SO_Z , I^\bullet  )  &\longrightarrow& \\
	\sigma_{Z*}  \bar{\pi}^*    \sHom_{\SO_{X}}( \pi_* \SO_Z , I^\bullet  ) &= & \sigma_{Z*} \pi^\flat I^\bullet.
	\end{array} \]
	It follows that we obtain a dualizing complex with involution $( \pi^\flat I^\bullet,\sigma_{\pi^\flat I^\bullet} )$ on $Z$. Thus, we have a well-defined spectrum $GW^{[i]} (Z,\sigma_Z, (\pi^{\flat} I^\bullet, \sigma_{\pi^{\flat}I}) ) $ for every $i \in \Z$.
	
	Note that if $Z = \Spec(B)$ and $X = \Spec(A) $ are affine, then $\Hom_A(B,I)$ has a structure of $B$-module by defining the action $b\cdot f $ as $b \cdot f : m \mapsto f(bm)$ for $b , m \in B$.\ 
	The involution $\sigma_{\pi^\flat I^\bullet} : \pi^\flat I^\bullet \rightarrow \sigma_{Z*} \pi^\flat I^\bullet$ is precisely the map of $B$-modules 
	\[ \sigma_{B,I}^A : \Hom_A(B,I) \rightarrow \Hom_A(B,I)^{\op_B} : f \mapsto \sigma_I f \sigma_B\].

	
	

	\section{The transfer morphism}
	
	\begin{theo}\label{thm:transfer}
	Let $(X,\sigma_X)$ and $(Z, \sigma_Z)$ be schemes with involution and  with $\frac{1}{2}$ in their global sections.  Suppose that $(X,\sigma_X)$ has a dualizing complex with involution $(I,\sigma_I)$ (cf. Definition \ref{Dwin}).\ If $ \pi: Z \rightarrow X$ is  a finite morphism of schemes with involution, then the direct image functor $\pi_*: \Ch^b_c(\Q(Z)) \rightarrow \Ch^b_c(\Q(X))  $ induces a map of spectra
\[T_{Z/X} : GW^{[i]} (Z,\sigma_Z, (\pi^{\flat} I^\bullet, \sigma_{\pi^{\flat}I}) ) \rightarrow GW^{[i]}(X,\sigma_X, (I^\bullet, \sigma_I) ). \] 
	\end{theo}
	
	\begin{Def}
		The map $T_{Z/X} : GW^{[i]} (Z,\sigma_Z, (\pi^{\flat} I^\bullet, \sigma_{\pi^{\flat}I}) ) \rightarrow GW^{[i]}(X,\sigma_X, (I^\bullet, \sigma_I) )$ in the above theorem is called the \textit{transfer morphism}. 
	\end{Def}
	\begin{proof}[Proof of Theorem \ref{thm:transfer}]
		The transfer map $T_{Z/X} : GW^{[i]} (Z,\sigma_Z, (\pi^{\flat} I^\bullet, \sigma_{\pi^{\flat}I}) ) \rightarrow GW^{[i]}(X,\sigma_X, (I^\bullet, \sigma_I) )$  is induced by the duality preserving functor given below in Lemma \ref{pushfo}.
	\end{proof}
	\begin{lem}\label{pushfo}
		Under the hypothesis of Theorem \ref{thm:transfer}, the direct image functor induces a duality preserving functor \[ \pi_*: (\Ch^b_c(\Q(Z)),  \#^{\pi^\flat I^\bullet}_{\sigma_Z} ) \longrightarrow (\Ch^b_c(\Q(X)),  \#^{I^\bullet}_{\sigma_X} ) \] with 
		the duality compatibility natural transformation 
		\[ \eta: \pi_* \circ  \#^{\pi^\flat I^\bullet}_{\sigma_Z} \longrightarrow   \#^{I^\bullet}_{\sigma_X}  \circ \pi_*, \] 
		given by the following composition of isomorphisms
		\[  \begin{CD} \eta_\GC : \pi_*  [ \sigma_{Z*} \GC ,  \pi^\flat I^\bullet  ]_{\SO_Z} @>\cong>> [  \pi_*  \sigma_{Z*} \GC , \pi_* \pi^\flat I^\bullet  ]_{\SO_X} \\  @>[p, 1]>\cong> [ \sigma_{X*} \pi_* \GC , \pi_* \pi^\flat I^\bullet  ]_{\SO_X} \\ @>[1, \xi]>\cong> [ \sigma_{X*} \pi_* \GC ,  I^\bullet  ]_{\SO_X}  \end{CD}  \]
		for $\GC \in \Ch^b_c(\Q(Z))$, where $ \xi: \pi_* \pi^\flat I^\bullet  \rightarrow I^\bullet $ is the evaluation at one\footnote{This map is called \textit{trace} in \cite{Ha66} and \cite{Gil07a}, but on affine schemes this map is in fact the evaluation at one.}.
	\end{lem}
	\begin{proof}
		We need to show that for any $M \in \Ch^b_c(\Q(Z))$ the following diagram in $\Ch^b_c(\Q(X))$ is commutative.
		\[    \begin{CD}\pi_*  M  @>>> \pi_*(M^{\#^{\pi^\flat I^\bullet}_{\sigma_Z} \#^{\pi^\flat I^\bullet}_{\sigma_Z} } )     \\ 
		@VVV  @VVV \\
		(\pi_*M)^{\#^{I^\bullet}_{\sigma_X} \#^{I^\bullet}_{\sigma_X}}   @>>> (\pi_*M^{\#^{\pi^\flat I^\bullet}_{\sigma_Z}} )^{\#^{I^\bullet}_{\sigma_X}} 
		\end{CD}     \]  
		 We write down what it means. To simplify the notation, in the rest of the proof we will drop the bullet in $I^\bullet$ and write $I$ instead. 
		
		\[ \xymatrix{\pi_*  M  \ar[r]^-{\pi_* \can_M^\sigma} \ar[d]_-{\can_{\pi_*M}^\sigma} & \pi_*
		[\sigma_{Z*}[\sigma_{Z*}M, \pi^\flat I], \pi^\flat I] \ar[d]^-{\eta_{[\sigma_{Z*}M, \pi^\flat I]}} \\
		[\sigma_{X*}[ \sigma_{X*} (\pi_*M), I ], I] 
		\ar[r]^-{[\eta_M, I]} & [\sigma_{X_*} \pi_*[\sigma_{Z*}M, \pi^\flat I], I]
	       }\] 
	
We depict all the maps by definition in Figure \ref{fig:dulaity} below. The diagram $\square_1$ in  Figure \ref{fig:dulaity} is commutative by the case of trivial duality (since all the maps considered here are natural, we drop the labeling of maps). We check the diagram $\square_2$ in Figure \ref{fig:dulaity} is also commutative.  This can be checked by the following two commutative diagrams.
	\begin{equation}\label{cd:1}
	\xymatrix{\sigma_{X*} \pi_*   [\sigma_{Z*} M, \pi^\flat I] \ar[d] \ar[r]&  [\sigma_{X*} \pi_*  \sigma_{Z*} M, \sigma_{X*} \pi_*  \pi^\flat I] \ar[r] & [\pi_*  \sigma_{Z*} \sigma_{Z*} M, \sigma_{X*} \pi_*  \pi^\flat I] \ar[d] \\
	\pi_*  \sigma_{Z*} [\sigma_{Z*} M, \pi^\flat I] \ar[r] & 	 [ \pi_*  \sigma_{Z*} \sigma_{Z*} M,  \pi_*  \sigma_{Z*} \pi^\flat I] \ar[r] &  [ \pi_*  M,  \pi_*  \sigma_{Z*} \pi^\flat I] }
	\end{equation}
	\begin{equation}\label{cd:2}
	\xymatrix{ 
	&\sigma_{X*} \pi_*  \pi^\flat I \ar[dd] \ar[rr] && \sigma_{X*} I \ar[rr] && I \ar@{=}[dd]
	\\
	\pi_*  \pi^\flat \sigma_{X*}  I \ar@{=}[dd] \ar[ur]\ar[rr] &&  \sigma_{X*} I \ar[rr] && I \ar@{=}[dd] \ar@{=}[ur]
	\\
    &\pi_*  \sigma_{Z*} \pi^\flat I \ar[rr] && \pi_*  \pi^\flat I   \ar[rr] && I
    \\
    	\pi_*  \pi^\flat \sigma_{X*}  I \ar[ur]\ar[rr] &&  \pi_*  \pi^\flat I  \ar[rr] && I \ar@{=}[ur]
    }\end{equation}
\begin{figure}[]
 	\[ 	 
 	 \begin{turn}{90}
		 \xymatrix@C=1em@R=6em{
		 	\pi_*  M \ar[rr] \ar[dd] && \pi_*
			[[M, \pi^\flat I], \pi^\flat I] \ar[d] \ar[r] & \pi_* [[M, \sigma_{Z*} \pi^\flat I], \pi^\flat I]  \ar[r] \ar[d]& \pi_* [\sigma_{Z*} [\sigma_{Z*} M, \pi^\flat I], \pi^\flat I]  \ar[d]  
			\\ 
			& \square_1 & [ \pi_*[M, \pi^\flat I],\pi_*\pi^\flat I] \ar[d]\ar[r]&  [ \pi_*[M, \sigma_{Z*} \pi^\flat I],\pi_*\pi^\flat I] \ar[r] \ar[d]& [\pi_*  \sigma_{Z*} [\sigma_{Z*} M, \pi^\flat I], \pi_*  \pi^\flat I]  \ar[d]
			\\
			[[ \pi_*M,  I ], I]
		      \ar[r] \ar[d] & [[ \pi_*M,\pi_*\pi^\flat I ], I] \ar[r]      & [ \pi_*[M, \pi^\flat I],I] \ar[r]&  [ \pi_*[M, \sigma_{Z*} \pi^\flat I], I] \ar[r] & [\pi_*  \sigma_{Z*} [\sigma_{Z*} M, \pi^\flat I], I]  \ar[dd]
	      \\
	   [[ \pi_*M, \sigma_{X*}  I ], I] \ar[d] && \square_2
	   \\
	   [ \sigma_{X*} [ \sigma_{X*} \pi_*M,   I ], I] \ar[rr] & & [ \sigma_{X*}    [ \pi_*\sigma_{Z*}[M,  I], I]  \ar[r]  & [ \sigma_{X*}    [ \pi_*\sigma_{Z*} M, \pi_* \pi^\flat I], I] \ar[r] & [ \sigma_{X*} \pi_*   [\sigma_{Z*} M, \pi^\flat I], I]
        }\end{turn}	\] 
    		\caption{Proof of Lemma \ref{pushfo}}\label{fig:dulaity}
    \end{figure}

Diagram (\ref{cd:1}) is commutative for obvious reasons. To see the commutativity of Diagram (\ref{cd:2}), we observe that the upper diagram is commutative by a variant of \cite[Proposition 6.6 p170]{Ha66}. 
The bottom diagram in Diagram (\ref{cd:2}) is also commutative by the definition of $\pi^\flat I \rightarrow \sigma_{Z*}  \pi^\flat I $ in Section \ref{sec:notation}. The left square is commutative by \cite[Proposition 6.3]{Ha66}). The front square is commutative by the naturality of the trace. Therefore, the back square is commutative which is what we need. Now, we apply the functor $[-,I]$ to  Diagram (\ref{cd:1}) and the functor $[[\pi_*M, -], I]$ to the back square diagram of Diagram (\ref{cd:2}). Combining these two diagrams gives a new one, which  by naturality and functoriality can be identified with $\square_2$. 
The remaining four small squares in Figure \ref{fig:dulaity} are commutative by naturality. 
 	\end{proof}
\begin{remark}\label{afftran}
	Let $\pi:R \rightarrow S$ be a finite morphism of rings with involution. We have constructed a duality preserving functor \[ \pi_*: (\sPerf (S \hMod),  \#^{\pi^\flat I}_{\sigma_S} ) \longrightarrow (\sPerf(R \hMod), \#^{I}_{\sigma_R} ) \] with 
	the duality compatibility  natural transformation 
	\[ \eta: \pi_* \circ  \#^{\pi^\flat I} \longrightarrow   \#^{I}  \circ \pi_*, \] 
	which is induced by the isomorphism of $R$-modules
	\begin{equation}\label{eao}
	\Hom_{S}(M^{\op_S}, \Hom_R(S,I)) \rightarrow \Hom_R(M^{\op_R}, I) : f \mapsto (\xi \circ f): m^{\op_R} \mapsto f(m^{\op_S}) (1_S) \end{equation} 
	for $M \in S\hMod $.
\end{remark}

	\section{The local case and finite length modules}
	Let $(R, \mathfrak{m},k)$ be a local ring with an involution $\sigma_R$.\ Suppose that $(R,\sigma_R)$ has a minimal dualizing complex with involution $(I^\bullet,\sigma_I)$. Note that this is the case if $R$ is Gorenstein local (cf. Lemma  \ref{lem:goe-dualizingcomplex}).  Suppose also that $\mathfrak{m}$ is invariant under the involution of $R$, i.e. $\sigma_R(\mathfrak{m}) = \mathfrak{m}$. Therefore, the $\sigma_R$ on $R$ induces an involution $\sigma_k$ on $k$. In the previous section, we have defined a transfer morphism
	\[T_{\mathfrak{m}/R} : W^{i} (k,\sigma_{k}, (\pi^{\flat} I^\bullet, \sigma_{\pi^{\flat}I}) ) \rightarrow W^{i}(R,\sigma_R, (I^\bullet, \sigma_I) )\]
	which factors through
	\[D_{\mathfrak{m}/R} : W^{i} (k,\sigma_{k}, (\pi^{\flat} I^\bullet, \sigma_{\pi^{\flat}I}) ) \rightarrow W^{i}_{\mathfrak{m}}(R,\sigma_R, (I^\bullet, \sigma_I) ).\]
	Moreover, the map $D_{\mathfrak{m}/R} : W^{i} (k,\sigma_{k}, (\pi^{\flat} I^\bullet, \sigma_{\pi^{\flat}I}) ) \rightarrow W^{i}_{\mathfrak{m}}(R,\sigma_R, (I^\bullet, \sigma_I) )$ factors as
	\[  W^{i} (k,\sigma_{k}, (\pi^{\flat} I^\bullet, \sigma_{\pi^{\flat}I}) )  \rightarrow W^{i}(\M_{fl}(R),\sigma_R, (I^\bullet ,\sigma_I)) \rightarrow W^{i}_{\mathfrak{m}}(R,\sigma_R, (I^\bullet, \sigma_I) ) \]
	where the first map is induced by the functor $\pi_*: \D^b(\M(k)) \rightarrow \D^b(\M_{fl}(R) ) $ and the second map is induced by the inclusion $\D^b(\M_{fl}(R) ) \rightarrow \D^b_{fl}(R) = \D^b_{\mathfrak{m}}(R)$. 
	\begin{theo}\label{thm:devissagelocal}
		Let $(R, \mathfrak{m},k)$ be a local ring with an involution $\sigma_R$.\ Suppose that $I^\bullet$ is a minimal dualizing complex of $R$, $\sigma_{I}$ is an involution of $I^\bullet$ and $\mathfrak{m}$ is invariant under the involution of $R$. Then, there is an isomorphism of groups
		\[D_{\mathfrak{m}/R} : W^{i} (k,\sigma_{k}, (\pi^{\flat} I^\bullet, \sigma_{\pi^{\flat}I}) ) \rightarrow W^{i}_{\mathfrak{m}}(R,\sigma_R, (I^\bullet, \sigma_I) ).\]
	\end{theo}
	
	\begin{lem}
		Under the assumption of Theorem \ref{thm:devissagelocal}, the natural map \[W^{i}(\M_{fl}(R), \sigma_R, (I^\bullet ,\sigma_I)) \rightarrow W^{i}_{\mathfrak{m}}(R,\sigma_R, (I^\bullet, \sigma_I))\] is an isomorphism of groups.
	\end{lem}
	\begin{proof}
		By \cite[Section 1.15]{Ke99} the inclusion 
		\[\D^b(\M_{fl}(R) ) \rightarrow \D^b_{fl}(R) = \D^b_{\mathfrak{m}}(R)\]
		is an equivalence of categories.
	\end{proof}
	\begin{lem}\label{lem:localdevissage}
	Under the assumption of Theorem \ref{thm:devissagelocal},
		the map \[\pi_*: W^{i} (k,\sigma_k, (\pi^{\flat} I^\bullet, \sigma_{\pi^{\flat}I}) ) \rightarrow W^{i}(\M_{fl}(R),\sigma_R, (I^\bullet ,\sigma_I))  \] is an isomorphism.
	\end{lem}
	\begin{proof} Let $I^\bullet: = 
		( \cdots \rightarrow 0 \rightarrow I^m \rightarrow I^{m+1} \rightarrow \cdots \rightarrow I^n \rightarrow 0 \rightarrow \cdots) $
		be a minimal dualizing complex on $R$. 
		We set $E : = I^n $ and let $\pi^\flat E $ be the $k$-module $ \Hom_R(k,E)$. Let 
		\[\sigma_{\pi^\flat E } :  \Hom_R(k,E) \rightarrow \Hom_R(k,E)^{\op_k}: f \mapsto \sigma_{\pi^\flat E } (f) \] be the involution on $\pi^\flat E$, where  $  \sigma_{\pi^\flat E } (f) (a) =  \sigma_E  f \sigma_k(a)$. Recall that $\sigma_{\pi^\flat E }$ is well-defined, and $(\pi^\flat E, \sigma_{\pi^\flat E }  )$ is a duality coefficient (cf. Definition \ref{Def:dualitycoefficient}) for $k\hMod$ (see Lemma \ref{lem: RSf}). Therefore, it induces a category with duality $(k\hMod, \#^{\pi^\flat E}_{\sigma_k})$ (recall Remark \ref{rmk:affineiden}). Moreover, the double dual identification 
		\[ \can_{\sigma_k, M}^{\pi^\flat E }: M \rightarrow \Hom_k(\Hom_k(M^{\op_k}, \pi^\flat E)^{\op_k}, \pi^\flat E): m \mapsto ( \can(m): f \mapsto \sigma_{\pi^\flat E} (f(m)) ) \]  is an isomorphism, therefore the duality is strong in the sense of \cite{Sch10b}.
		
		Now, consider the following commutative diagram
		\[ \xymatrix{ W^{i+n}\big(k,\sigma_k,(\pi^\flat E, \sigma_{\pi^\flat E})\big) \ar[r]^-{\cong}_-{h}  \ar[d]^-{\cong}_-{v_1} & W^{i+n}(\M_{fl}(R),\sigma_R, (E, \sigma_E)) \ar[d]^-{\cong}_-{v_2}\ \  \\
			W^{i} (k,\sigma_k, (\pi^{\flat} I^\bullet, \sigma_{\pi^{\flat}I}) )    \ar[r]_-{\pi_*} & W^{i}(\M_{fl}(R),\sigma_R, (I^\bullet ,\sigma_I))\ \ .  } \]
		Here the map $h$ is the transfer morphism, and by \cite[Proposition 5.1]{BW02} it is an isomorphism for $i+n$ even, see also \cite[Corollary 6.9 and Theorem 6.10]{QSS79} or  (\cite[Theorem 6.1]{Sch10a}). Note that the categories involved in the diagram above are all  derived categories of abelian categories, and if $i+n$ is even, the derived Witt groups can be identified with the usual Witt groups of symmetric or skew-symmetric forms. If $i+n$ is odd by \cite[Proposition 5.2]{BW02} one has
		\[ W^{i+n}\big(k,\sigma_k,(\pi^\flat E, \sigma_{\pi^\flat E})\big) = W^{i+n}(\M_{fl}(R),\sigma_R, (E, \sigma_E)) = 0.\]
		The map $ v_1 $ is induced by the duality preserving functor
		\[(\id, \theta) :  (\D^b k , \#_{\sigma_k}^{\pi^{\flat}E[n]} ) \rightarrow ( \D^b k, \#^{\pi^{\flat} I^\bullet}_{\sigma_k} ), \] 
		where $\theta: \#_{\sigma_k}^{\pi^{\flat}E[n]} \rightarrow \#^{\pi^{\flat} I^\bullet}_{\sigma_k} $ is the natural isomorphism given by the isomorphism of complexes
		\[ \theta_M :  \Hom_k(M^{\op_k}, \Hom_R(k,E[n]) )  \rightarrow \Hom_k(M^{\op_k}, \Hom_R(k,I^\bullet) )  \] 
		which is induced by the inclusion of complexes $ E[n] \rightarrow I^\bullet$. Notice that $ \theta_M$ is an isomorphism, because as an $R$-module
		\[  \Hom_k(M^{\op_k}, \Hom_R(k,I^j) ) \cong \Hom_R (M^{\op_R}, I^j )=0 \] 
		if $j \neq n$. The first isomorphism follows by the isomorphism $\eta$ in (\ref{eao}) and the second identity can be proved by \cite[Lemma 3.3]{Gil02}. Indeed, the local ring $R$ is assumed to be Gorenstein in \cite[Lemma 3.3 (1)-(4)]{Gil02}. We need to explain why \cite[Lemma 3.3]{Gil02} can be applied to our situation that the local ring $R$ only assumed to have a minimal dualizing complex. Note that there is an isomorphism $I^j \cong \oplus_{\mu_I(Q) = j} E_R(R/Q) $ by \cite[Theorem 1.15]{Gil07a} where $\mu_I: \Spec~R \rightarrow \Z$ is the codimension function of the dualizing complex $\I^\bullet$ defined after \cite[Lemma 1.12]{Gil07a}. By \cite[Lemma 1.14]{Gil07a}, the maximal ideal $\mathfrak{m}$ is the only prime ideal $Q$ such that $\mu_I(Q) =n$.  For any $R$-module $M$ such that $ \mathfrak{m}M =0$, we conclude that $\Hom_R(M, E_R(R/Q)) = 0$ if $Q \subsetneq \mathfrak{m} $ (i.e. $\mu_I (Q) < n$) by  \cite[Lemma 3.3 (5)]{Gil02} where the local ring is only assumed to be commutative Noetherian.\ 
		
		The isomorphism $v_2$ is defined analogously to the map $v_1$.
	\end{proof}
	\begin{remark}
	  For regular local rings Lemma \ref{lem:localdevissage}  was proved by \cite[Theorem 4.5]{Gil09}. The situation in \cite[Section 3.1]{Gil07a} is very similar to Lemma \ref{lem:localdevissage}, but in the construction of the transfer map we avoid the choice of an embedding 
	  of $k$ into its injective hull $E = E_R(k)$ over $R$.  
	\end{remark}
	\begin{lem}\label{localcase}
		Let $(R, \mathfrak{m},k)$ be a local ring with involution and $\frac{1}{2} \in R$.  Assume that $I^\bullet$ is a minimal dualizing complex of $R$, $\sigma_{I}$ is an involution of $I^\bullet$ and  $\mathfrak{m}$ is invariant under the involution of $R$. Let $J$ be an invariant ideal of $R$ and consider the following canonical projections
		\[ \xymatrix{ R \ar[d]^-{p} \ar[r]^-{\pi} & k \\
			R/J \ar[ur]_-{q}&\ \ . } \]
		We have a commutative diagram
		\begin{equation}\label{WmJ} \xymatrix{  
			W^{i} (k,\sigma_k, (q^{\flat} p^{\flat} I^\bullet, \sigma_{ q^{\flat} p^{\flat}I}) )   \ar[d]^-{\cong}   \ar[r]^-{q_*} & W^{i}_{\mathfrak{m}/J}(R/J,\sigma_{R/J}, (p^\flat I^\bullet, \sigma_{p^{\flat}I} ))\ar[d]^-{p_*}   \\
			W^{i} (k,\sigma_k, (\pi^{\flat} I^\bullet, \sigma_{\pi^{\flat}I}) )    \ar[r]^-{\pi_*} & W^{i}_{\mathfrak{m}}(R,\sigma_R, (I^\bullet, \sigma_I))     } \end{equation}
		Therefore, the map $p_*$ is an isomorphism
	\end{lem}
	
	\begin{proof}
		Define a map $ \gamma: q^{\flat} p^{\flat} I  \rightarrow  \pi^\flat I  $ 
		\[ \gamma: \Hom_{R/J}(k, \Hom_R(R/J, I )) \rightarrow \Hom_R(k,I)  \]
		by sending
		$ f: k\rightarrow \Hom_R(R/J, I )  $ to $ \gamma(f): a \mapsto f(a)(1) $. This map is an isomorphism, cf. \cite[Proposition 6.2 p166]{Ha66}. 
		The commutativity of Diagram (\ref{WmJ}) follows immediately from the following commutative diagram of triangulated categories with duality
		\[ \xymatrix{  ( \D^b k,\#_{\sigma_k}^{q^{\flat} p^{\flat}I  }  )   \ar[d]     \ar[r]^-{q_*} &  (\D^b_{\mathfrak{m}/J} (R/J), \#_{\sigma_{R/J}}^{p^{\flat}I} )  \ar[d]^-{p_*} &  \\
			( \D^b k,\#_{\sigma_k}^{\pi^{\flat} I})   \ar[r]^-{\pi_*} & (\D^b_{\mathfrak{m}} R, \#_{\sigma_R}^{I})&\hspace{- 1.7 cm},                } \]
		where the left vertical arrow is induced by the isomorphism $\gamma: q^{\flat} p^{\flat} I  \rightarrow  \pi^\flat I $. 
	\end{proof}
	\section{The d\'{e}vissage theorem for a closed immersion}
	Let $(X,\sigma_X)$ and $(Z, \sigma_Z)$ be schemes with involution and with $\frac{1}{2} \in \SO_X$.  Suppose that  $(X,\sigma_X)$ has a minimal dualizing complex with involution $(I,\sigma_I)$  (cf. Definition \ref{Dwin}).\  Assume that $ \pi: Z \hookrightarrow X$ is a closed immersion which is invariant under involutions. Since $\pi_* \GC \in \Ch^b_{c,Z}(\Q(X))  $ for all $\GC \in \Ch^b_c(\Q(Z)) $, we have a map  
	\[ D_{Z/X} :  GW^{[i]} (Z,\sigma_Z, (\pi^{\flat} I^\bullet, \sigma_{\pi^{\flat}I}) ) \rightarrow GW^{[i]}_Z(X,\sigma_X, (I^\bullet, \sigma_I) ). \] 
	In fact, the transfer morphism $T_{Z/X}$ is the composition 
	\[ GW^{[i]} (Z,\sigma_Z, (\pi^{\flat} I^\bullet, \sigma_{\pi^{\flat}I}) ) \stackrel{D_{Z/X}}\longrightarrow GW^{[i]}_Z(X,\sigma_X, (I^\bullet, \sigma_I))  \rightarrow GW^{[i]}(X,\sigma_X, (I^\bullet, \sigma_I) ),  \]
	where the second map is extending the support.
	\begin{theo}\label{thm:devissage}
	Let $(X,\sigma_X)$ and $(Z, \sigma_Z)$ be schemes with involution and with $\frac{1}{2} \in \SO_X$. Suppose that  $(X,\sigma_X)$ has a a minimal dualizing complex with involution $(I,\sigma_I)$.\ If $ \pi: Z \hookrightarrow X$ is a closed immersion which is invariant under involutions, then the direct image functor $\pi_*: \Ch^b_c(\Q(Z)) \rightarrow \Ch^b_c(\Q(X))  $ induces an equivalence of spectra
\[ D_{Z/X} : GW^{[i]} (Z,\sigma_Z, (\pi^{\flat} I^\bullet, \sigma_{\pi^{\flat}I}) ) \rightarrow GW^{[i]}_Z(X,\sigma_X, (I^\bullet, \sigma_I) ). \] 	
	\end{theo}
	\begin{proof}
		The result follows by Theorem \ref{DevisW} and Karoubi induction \cite[Lemma 6.4]{Sch17}.
	\end{proof}
	\begin{theo}\label{DevisW}
		Under the hypothesis of Theorem \ref{thm:devissage}, there is an isomorphism of groups
		\[ D_{Z/X} : W^{i} (Z,\sigma_Z, (\pi^{\flat} I^\bullet, \sigma_{\pi^{\flat}I}) ) \rightarrow W^{i}_Z(X,\sigma_X, (I^\bullet, \sigma_I) ). \] 
	\end{theo}
	\begin{proof} We  borrow the idea of \cite{BW02}, \cite{Gil02}, \cite{Gil07a} to use the filtration of derived category  by codimension of points to reduce the problem to the local case. All we need to check is that this approach is compatible with the duality given by a non-trivial involution.
		
		Let $X^{p}_{\mathcal{I}}: = \{ x \in X | m \leq \mu_{\mathcal{I}} (x) \leq p-1 \} $ and $X^{(p)}_{\mathcal{I}}: = \{ x \in X | \mu_{\mathcal{I}} (x) = p \} $. We define 
		\[ \D^p_{Z,X} : = \Big\{ M^\bullet \in \D^b_{c,Z}(\mathcal{Q}(X))| (M^\bullet)_x \textnormal{ is acyclic for all $x \in X^{p}_\mathcal{I}$}  \Big\}\]
		considered as a full subcategory of $ \D^b_{c,Z}(\mathcal{Q}(X))$. Note that the duality functor $ \#^{I^\bullet}_{\sigma_X}$ maps $\D^p_{Z,X} $ into itself. Therefore, $( \D^p_{Z,X},\#^{I^\bullet}_{\sigma_X})$ is a triangulated category with duality. The subcategories $\D^p_{Z,X} $ provide a finite filtration
		\[ \D^b_{c,Z}(\mathcal{Q}(X)) =   \D^m_{Z,X} \supseteq \D_{Z,X}^{m+1} \supseteq \cdots \supseteq \D_{Z,X}^p \supseteq \cdots \supseteq \D_{Z,X}^n  \supseteq (0) \]
		which induces exact sequences of triangulated categories with duality
		\[ \D_{Z,X}^{p+1} \longrightarrow \D_{Z,X}^{p} \longrightarrow \D_{Z,X}^{p}/ \D_{Z,X}^{p+1}.   \]
		On the other hand, we define
		\[\D_Z^{p}: = \Big\{ M^\bullet \in \D^b_{c}(\mathcal{Q}(Z))| (M^\bullet)_z \textnormal{ is acyclic for all $z \in Z^{p}_{\pi^\flat\mathcal{I}}$} \Big\}  \]
		as a full subcategory of $ \D^b_{c}(\mathcal{Q}(Z))$. By the same reason above, we have a finite filtration
		\[ \D^b_{c}(\mathcal{Q}(Z)) =   \D^m_Z \supseteq \D_Z^{m+1} \supseteq \cdots  \supseteq \D_{Z}^p \supseteq \cdots \supseteq \D_Z^n \supseteq (0) \]
		which induces exact sequences of triangulated categories with duality
		\[ \D_{Z}^{p+1} \longrightarrow \D_{Z}^{p} \longrightarrow \D_{Z}^{p}/ \D_{Z}^{p+1}.   \]
		Since $\mu_{\mathcal{I}} (z) = \mu_{ \pi^\flat \mathcal{I}} (z)$ for all $z\in Z$, we have $\pi_*( \D_{Z}^{p}) \subset \D_{Z,X}^{p} $. It follows that we obtain a map of exact sequences of triangulated categories with duality
		\[    \begin{CD} (\D_{Z}^{p+1}, \#^{\pi^\flat I^\bullet}_{\sigma_Z})  @>>> (\D_{Z}^{p}, \#^{\pi^\flat I^\bullet}_{\sigma_Z}) @>>>  (\D_{Z,X}^{p}/ \D_{Z,X}^{p+1} , \#^{\pi^\flat I^\bullet}_{\sigma_Z})  \\ 
		@V{\pi_*}VV  @V{\pi_*}VV  @V{\pi_*}VV \\
		(\D_{Z,X}^{p+1} , \#^{I^\bullet}_{\sigma_X}) @>>>  (\D_{Z,X}^{p}, \#^{I^\bullet}_{\sigma_X}) @>>> (\D_{Z,X}^{p}/ \D_{Z,X}^{p+1}, \#^{I^\bullet}_{\sigma_X})
		\end{CD}     \] 
		which induces a map of long exact sequences of groups
		\begin{equation}\label{map}  \small{  \begin{CD} \cdots  \longrightarrow & W^i  (\D_{Z}^{p}, \#^{\pi^\flat I^\bullet}_{\sigma_Z})  & \longrightarrow & W^i (\D_{Z}^{p}/ \D_{Z}^{p+1} , \#^{\pi^\flat I^\bullet}_{\sigma_Z} )&\longrightarrow& W^{i+1} (\D_{Z}^{p+1}, \#^{\pi^\flat I^\bullet}_{\sigma_Z})   & \longrightarrow \cdots &\\ 
			& @V{\pi_*}VV  @V{\pi_*}VV  @V{\pi_*}VV  \\
			\cdots \longrightarrow & W^i   (\D_{Z,X}^{p}, \#^{I^\bullet}_{\sigma_X}) &\longrightarrow& W^i  (\D_{Z,X}^{p}/ \D_{Z,X}^{p+1}, \#^{I^\bullet}_{\sigma_X})  &\longrightarrow& W^{i+1}  (\D_{Z,X}^{p+1} , \#^{I^\bullet}_{\sigma_X}) & \longrightarrow \cdots&\phantom{a} .
			\end{CD}   }  \end{equation}
		\begin{lem}
			The localization functors
			\[ loc: \D_{Z,X}^{p}/ \D_{Z,X}^{p+1} \rightarrow \prod_{x \in Z \cap X^{(p)}_{\mathcal{I}}} \D^b_{\mathfrak{m}_{X,x}}(\SO_{X,x}) \]
			and \[loc: \D_{Z}^{p}/ \D_{Z}^{p+1} \rightarrow \prod_{z \in Z^{(p)}_{\pi^\flat \mathcal{I}}} \D^b_{\mathfrak{m}_{Z,z}}(\SO_{Z,z}) \]
			induce equivalences of categories.  
		\end{lem} 
		\begin{proof}
			This result is well-known in the literature, see \cite[Theorem 5.2]{Gil07b} for instance. 
		\end{proof}
		Note that the morphism $\pi_*: W^i (\D_{Z}^{p}/ \D_{Z}^{p+1} , \#^{\pi^\flat I^\bullet}_{\sigma_Z} ) \rightarrow W^i  (\D_{Z,X}^{p}/ \D_{Z,X}^{p+1}, \#^{I^\bullet}_{\sigma_X}) $ is an isomorphism. This can be concluded by the commutative diagram
		\[ \begin{CD} W^i (\D_{Z}^{p}/ \D_{Z}^{p+1} , \#^{\pi^\flat I^\bullet}_{\sigma_Z} ) @>loc>\cong> \bigoplus_{z = \sigma(z)\in  Z^{(p)}_{\pi^\flat \mathcal{I}} }  W^i_{\mathfrak{m}_{Z,z}} ( \SO_{Z,z}, \sigma_{\SO_{Z,z}}, (\pi^\flat \mathcal{I}_z, \sigma_{\pi^\flat \mathcal{I}_z}))\\
		@VV\pi_*V @V\cong V{\oplus_{z} (D_{ \SO_{Z,z} /  \SO_{X,z} } )}V \\
		W^i  (\D_{Z,X}^{p}/ \D_{Z,X}^{p+1}, \#^{I^\bullet}_{\sigma_X})  @>loc>\cong> \bigoplus_{z = \sigma(z) \in Z \cap X^{(p)}_{\mathcal{I}}}  W^i_{\mathfrak{m}_{X,z}} ( \SO_{X,z}, \sigma_{\SO_{X,z}}, ( \mathcal{I}_z, \sigma_{ \mathcal{I}_z})),  \end{CD} \]
		where the right morphism is an isomorphism by Lemma \ref{localcase}. Note that the Witt groups on the right hand side do not contain the component $\sigma(z) \neq z$. To explain, we write down a duality preserving  functor 
		\[(loc,\eta): (\D_{Z}^{p}/ \D_{Z}^{p+1} \#^{\pi^\flat I^\bullet}_{\sigma_Z} ) \rightarrow \Big(\D^b_{\mathfrak{m}_{Z,z}}(\SO_{Z,z}) \times \D^b_{\mathfrak{m}_{Z,\sigma(z)}}(\SO_{Z,\sigma(z)}), \# \Big),  \]
		where $\#: (M,N) \mapsto (N^{\#^{(\pi^\flat I)_z}_{\sigma_z}}, M^{\#^{(\pi^\flat I)_{\sigma(z)}}_{\sigma_{\sigma(z)}}})$ and the duality compatibility isomorphism \[\eta_M : ( (M^{\#^{(\pi^\flat I)}_{\sigma_Z}})_z, (M^{\#^{(\pi^\flat I)}_{\sigma_Z}})_{\sigma(z)} ) \longrightarrow ( M^{\#^{(\pi^\flat I)_z}_{\sigma_z}}_{\sigma(z)}, M^{\#^{(\pi^\flat I)_{\sigma(z)}}_{\sigma_{\sigma(z)}}}_{z} )   \]
		is defined to be the canonical isomorphism induced by restrictions of scalars via the local isomorphisms of local rings $\sigma: \SO_{Z,z} \rightarrow \SO_{Z,\sigma(z)}$ and $\sigma^{-1}: \SO_{Z, \sigma(z)} \rightarrow \SO_{Z,z}$. By these local isomorphisms, note that $\D^b_{\mathfrak{m}_{Z,z}}(\SO_{Z,z})$ is equivalent to $\D^b_{\mathfrak{m}_{Z,\sigma(z)}}(\SO_{Z,\sigma(z)})$, and \[W^i (\D^b_{\mathfrak{m}_{Z,z}}(\SO_{Z,z}) \times \D^b_{\mathfrak{m}_{Z,\sigma(z)}}(\SO_{Z,\sigma(z)}), \# ) =0\] by a version of the additivity theorem (see \cite[Proposition 6.8]{Sch17}, especially the proof). 
		 
	 Finally, by induction on $p$ in Diagram (\ref{map}), we conclude the result. 
	\end{proof}

	\section{A geometrical computation of the d\'{e}vissage for regular schemes}
	Let $(X,\sigma_X)$ be a regular scheme with involution. Assume $(Z,\sigma_Z)$ is a regular scheme with involution and assume further that $Z$ is regularly embedded in $X$ of codimension $d$. We have a normal bundle $N:= N_{Z/X}$ of $X$ on $Z$ which is locally free of rank $d$.  If $Z$ is invariant under $\sigma_X$, then we can define an involution $\sigma_N: N \rightarrow \sigma_* N $ induced from $\sigma_X: X \rightarrow X$ by the invariant ideal sheaf of $Z$ in $X$. The involution $\sigma_N$ induces an involution $\det(\sigma_N): \det(N) \rightarrow \sigma_* \det(N)  $ by taking determinants. Recall that the canonical sheaf is defined by
	\[ \omega_{Z/X} := \sHom_{\SO_Z}(\det(N), \SO_Z). \]
	Now, $\det(\sigma_N)$ induces an involution
	\[ \sigma_{\omega_{Z/X}} : \omega_{Z/X} \rightarrow  \sigma_*\omega_{Z/X} : f \mapsto  \sigma_Z  \circ (\sigma_*f) \circ \det(\sigma_N).\]
	\begin{theo}[D\'evissage]\label{thm:Devissage} Let $(X,\sigma_X)$ be a regular scheme with involution.  Let $(\LC,\sigma_\LC)$ be a dualizing coefficient with $\LC$ a locally free $\SO_X$-module of rank one. If $Z$ is a regular scheme regularly embedded in $X$ of codimension $d$ which is invariant under $\sigma$, then there is an equivalence of spectra
		\[ D_{Z/X, \mathcal{L}} : GW^{[i-d]} (Z,\sigma_Z, (\omega_{Z/X} \otimes_{\SO_{X}} \mathcal{L}, \sigma_{\omega_{Z/X}} \otimes  \sigma_{\mathcal{L} }) )  \longrightarrow GW^{[i]}_Z(X,\sigma_X, (\mathcal{L}, \sigma_{\mathcal{L}}) ). \]
	\end{theo}
	\begin{proof}  By Lemma \ref{lem:goe-dualizingcomplex} there exists $(\mathcal{L},\sigma_{\mathcal{L}}) \rightarrow (\I^\bullet,\sigma_I)$, a minimal injective resolution compatible with $\sigma_X$, where $\I^\bullet =  \I^0 \rightarrow \I^1 \rightarrow \cdots \rightarrow \I^n $. We have the d\'{e}vissage isomorphism on the level of coherent Grothendieck-Witt groups
		\[ D_{Z/X} : GW^{i} (Z,\sigma_Z, (\pi^{\flat} \I^\bullet, \sigma_{\pi^{\flat}\I}) ) \rightarrow GW^{i}_Z(X,\sigma_X, (\I^\bullet, \sigma_\I) ). \] 
		We want to construct a quasi-isomorphism
		$\beta:\omega_{Z/X} \otimes_{\SO_{X}} \mathcal{L}[-d] \stackrel{\simeq}\longrightarrow \pi^\flat \I^\bullet$ which is compatible with involutions, then we can use a variant of Lemma \ref{II'} to conclude the result. 
		 We will see that this map stems from  
		the fundamental local isomorphism
		\[ H^i \sHom_{\SO_X}(\SO_Z, I^\bullet) =  \Ext^i_{\SO_{X}} (\SO_Z, \mathcal{L}) = \left\{   \begin{array}{lll}
		0 & \textnormal{if $i \neq d$ } \\
		\omega_{Z/X} \otimes_{\SO_{X}} \mathcal{L} & \textnormal{if $i = d$ }
		\end{array} \right. \]
		(see \cite[p 179]{Ha66}). Before checking the compatibility of the involutions,  we review the fundamental local isomorphism. We take  an open affine subscheme $U\cong \Spec(R)$ of $X$, and $Z_U := Z \times_X U \cong \Spec(R/J)$ with $J$ defined by a regular sequence $(x_1 , \cdots , x_d)$ of length $d$. Let $L :=\LC_U$ and $I:=\I_U$. Let $E = \oplus_{ 1 \leq i \leq d} Re_i $ be a free $R$-module with basis $\{ e_1, \cdots ,e_d\}$. 
			Let $s:E \rightarrow R :  (y_i e_i)_{1\leq i \leq d} \mapsto \sum_{i=1}^d y_i x_i $. There exists a Koszul resolution of $R/J$  
		\[ \cdots \rightarrow 0 \rightarrow \extp^d E \rightarrow \extp^{d-1}E \rightarrow \cdots \rightarrow \extp^0 E \rightarrow R/J  \]
		with differentials given by 
		\[d^i : \extp^i E \rightarrow \extp^{i-1} E : \alpha_1 \wedge \cdots \wedge \alpha_i \mapsto \sum_{t=1}^{i} (-1)^{t+1} s(\alpha_t) \alpha_1 \wedge \cdots \wedge \widehat{\alpha_t} \wedge \cdots \wedge \alpha_i. \]
		Since $(x_1 , \cdots , x_d)$ is a regular sequence, the Koszul complex provides a  projective resolution of $R/J$. Set $P^{-i} = P_i = \extp^i E $.
		
	 The fundamental local isomorphism 
		\[ \ext^i_R (R/J, L )  \rightarrow \Hom_R(\det(J/J^2), L)  \]
		of the $R/J$-modules (see \cite[Chapter III.7 p 176]{Ha66}) is defined as follows. We draw the following diagram of $R$-modules
		\[\arraycolsep=1.7pt\def\arraystretch{1.7}
		\begin{array}{cc|ccccc}
		\multicolumn{1}{c}{  }& \multicolumn{1}{c}{  }  & \multicolumn{1}{c}{  } & \multicolumn{1}{c}{  } & \multicolumn{1}{c}{  } & \multicolumn{1}{c}{  } & \multicolumn{1}{c}{ \Hom_R(\det(J/J^2), L/JL) } \\
		\multicolumn{1}{c}{  }& \multicolumn{1}{c}{  }  & \multicolumn{1}{c}{  } & \multicolumn{1}{c}{  } & \multicolumn{1}{c}{  } & \multicolumn{1}{c}{  } & \multicolumn{1}{c}{ \tilde{\beta} \xuparrow{0.35cm} } \\
		\multicolumn{1}{c}{  }& \multicolumn{1}{c}{   }  & \multicolumn{1}{c}{\Hom_R(P^0, L)} & \multicolumn{1}{c}{ \longrightarrow } & \multicolumn{1}{c}{ \cdots } & \multicolumn{1}{c}{ \longrightarrow } & \multicolumn{1}{c}{ \Hom_R(P^d, L) }  \\
		\multicolumn{1}{c}{  }& \multicolumn{1}{c}{  }  & \multicolumn{1}{c}{ \xdownarrow{0.35cm} } & \multicolumn{1}{c}{  } & \multicolumn{1}{c}{  } & \multicolumn{1}{c}{  } & \multicolumn{1}{c}{ \xdownarrow{0.35cm} } \\
		\cline{3-7} 
		\Hom_R(R/J, I^0) & \longrightarrow & \Hom_R(P^0, I^0) & \longrightarrow &\cdots & \longrightarrow  & \Hom_R(P^d, I^0) \\
		\xdownarrow{0.35cm}
		&   &  \xdownarrow{0.35cm}
		&&&&   \xdownarrow{0.35cm}
		\\
		\Hom_R(R/J, I^1) & \longrightarrow & \Hom_R(P^0, I^1) & \longrightarrow &\cdots & \longrightarrow  & \Hom_R(P^d, I^1) \\
		\xdownarrow{0.35cm}
		&   &  \xdownarrow{0.35cm}
		&&&&   \xdownarrow{0.35cm}\\
		\Hom_R(R/J, I^2) & \longrightarrow & \Hom_R(P^0, I^2) & \longrightarrow &\cdots & \longrightarrow  & \Hom_R(P^d, I^2) \\
		\xdownarrow{0.35cm}
		&   &  \xdownarrow{0.35cm}
		&&&&   \xdownarrow{0.35cm}
		\\
		\vdots &   & \vdots &   &\vdots &   & \vdots \\
		\xdownarrow{0.35cm}
		&   &  \xdownarrow{0.35cm}
		&&&&   \xdownarrow{0.35cm} \\
		\Hom_R(R/J, I^n) & \longrightarrow & \Hom_R(P^0, I^n) & \longrightarrow &\cdots & \longrightarrow  & \Hom_R(P^d, I^n) 
		\end{array}\]
		
		Note that by this diagram, we have zigzags 
		\[\Hom_{R}(R/J,I^\bullet) \longrightarrow \Tot \big(\Hom_R(P^\bullet, I^\bullet)\big)  \longleftarrow \Hom_R(P^\bullet, L)  \]
		of quasi-isomorphisms of $R$-modules. Now, we deduce 
		\[ \ext^ i_R (R/J, L) := H^i(\Hom_{R}(R/J,I^\bullet)) \cong H^i\Hom_R(P^\bullet, L)  \]
		and we define a map of $R$-modules
		\[\tilde{\beta} :   \Hom_R(\det E, L) \rightarrow \Hom_R(\det(J/J^2), L/JL):  f \mapsto \tilde{\beta}(f),    \]
		where $  \tilde{\beta} (f) (\bar{x}_1 \wedge \cdots \wedge \bar{x}_d ) :=  f(x_1 \wedge \cdots \wedge x_d)$. One checks that the composition
		\[ \Hom_R(P^{d-1}, L) \rightarrow \Hom_R(P^d, L) \stackrel{\beta}\rightarrow  \Hom_R(\det(J/J^2), L/JL) \]
		is zero. Thus, we get a map of $R$-modules
		\[ H^d\Hom(P^\bullet, L) \rightarrow  \Hom_R(\det(J/J^2), L/JL)  \]
		which can be considered as a map of $R/J$-modules, and it is an isomorphism. This morphism extends to a global morphism as it does not depend on the choice of the regular sequence (cf. \cite[Chapter III.7 Lemma 7.1]{Ha66}).
		
		Now, we check that $\beta$ is compatible with the involution. This can be checked affine locally. 
		We take any open affine subscheme $U \cong \Spec (R)$ of $X$, which need not to be compatible with the involution. It follows that the subscheme $ U' :=\sigma_X(U)$ is also affine and can be assumed to be isomorphic to $\Spec(R')$ for some ring $R'$. Then the involution $\sigma_X$ restricts to a pair of isomorphisms of affine schemes $\sigma_U : U \rightarrow U' $ and $\sigma_{U'} : U' \rightarrow U $ (with $\sigma_{U'} = \sigma_U^{-1}$) which induces a pair of isomorphisms of rings $\sigma_R: R \rightarrow R'$ and $\sigma_{R'}: R' \rightarrow R$ (with $\sigma_{R'} = \sigma_R^{-1}$). Recall that $Z_U \cong \Spec(R/J)$, where $J $ is an ideal of $R$ defined by a regular sequence $(x_1 ,\ldots, x_d)$ of length $d$.   Consider the restriction of $U'$ into the closed subscheme $Z_U' $ which is isomorphic to $\Spec(R'/J')$ for $J'$ an ideal inside $R'$. Since by assumption $Z$ is invariant under the involution of $X$, we have the following commutative diagram of $R$-modules
		\[ \xymatrix{
		J \ar[r] \ar[d]^-{\sigma_J}& R \ar[r] \ar[d]^-{\sigma_R} & R/J \ar[d]^-{\sigma_{R/J}} \\
	    J' \ar[r] \ar[d]^-{\sigma_{J'}}& R' \ar[r] \ar[d]^-{\sigma_{R'}} & R'/J' \ar[d]^-{\sigma_{R'/J'}} \\
	    J \ar[r] & R \ar[r]  & R/J\ ,  \\}\] where the vertical arrows are all isomorphisms and $\sigma_J = \sigma_{J'}^{-1}, \sigma_R = \sigma_{R'}^{-1}, \sigma_{R'/J'} = \sigma_{R/J}^{-1}$. Therefore, the sequence $(\sigma_R(x_1), \cdots, \sigma_R(x_d))$ is a regular sequence in $R'$ defining $J'$.

		As above, let $L' :=\LC_{U'}$ and $I':=\I_{U'}$. Let $E' = \oplus_{ 1 \leq i \leq d} Re'_i $ be a free $R'$-module with basis $\{ e'_1, \cdots ,e'_d\}$. Similarly, we have a Koszul complex $P'^\bullet \rightarrow R'/J' $ associated to the section $s': E' \rightarrow R': (y_i e_i)_{1\leq i \leq d} \mapsto \sum_{i=1}^d y_i \sigma(x_i)$ with $P'^{-i}: = \Lambda^i E'$.\ We now consider the following isomorphism of $R$-modules
		\[\sigma_E: E \rightarrow \sigma_{R*} E' : (y_i e_i)_{1\leq i \leq d} \mapsto (\sigma_R(y_i) e_i)_{1\leq i \leq d} \]
		It induces isomorphisms of $R$-modules
		\[\extp^i \sigma_E : \extp^i E\rightarrow \sigma_{R*}(\extp^i E')\] 
		on the exterior powers $\extp^i E$. It is straightforward to check that the diagram
		\[ \begin{CD}
		E @>s>> R \\
		@V\sigma_EVV @V\sigma_RVV \\
		\sigma_{R*}E' @>s'>> \sigma_{R*} R'
		\end{CD}\]
		of $R$-modules commutes. By checking the differentials, we get an isomorphism $ \sigma_P: P^\bullet \rightarrow \sigma_{R*}P'^\bullet $ of complexes of $R$-modules with inverse $ \sigma_{P'}: P'^\bullet \rightarrow \sigma_{R'*}P^\bullet $. Moreover, we can verify that the diagram
		\[ \xymatrix{
		       \Hom_R(P^d, L) \ar[r]^-{\beta} \ar[d]^-{\epsilon_{P^d,L}} & \Hom_R(\det(J/J^2), L/JL)  \ar[d]^-{\epsilon_{\det(J/J^2), L/JL}} \\
	       \Hom_{R'}(P'^d, L') \ar[r]^-{\beta'}  & \Hom_{R'}(\det(J'/J'^2), L'/J'L') } 
	\]
		is also commutative.
	
		 Using the same notations as in Lemma \ref{lem: RSf}, we form the following commutative diagram
		\[
		\xymatrix@C=1em{
		\Hom_{R}(R/J, I^\bullet)  \ar[r]  \ar[d]^-{\epsilon_{R/J,I}}& \Tot \big(\Hom_R(P^\bullet, I^\bullet)\big) \ar[d]^-{\epsilon_{P,I}} & \ar[l]  \Hom_R(P^\bullet, L) \ar[d]^-{\epsilon_{P,L}} \ar[r] &  \Hom_R(\det(J/J^2), L/JL) [-d] \ar[d]^-{\epsilon_{\det(J/J^2), L/JL}}\\
		\Hom_{R'}(R'/J', I'^\bullet) \ar[r]& \Tot \big(\Hom_{R'}(P'^\bullet, I'^\bullet)\big)  &\ar[l] \Hom_{R'}(P'^\bullet, L') \ar[r] &  \Hom_{R'}(\det(J'/J'^2), L'/J'L') [-d]
	}\]
		of complexes of $R$-modules. 
		Since the horizontal maps are all quasi isomorphisms and they induce maps on homology, we get   
		\[\begin{CD}
		H^d \Hom_{R}(R/J, I^\bullet) @>\cong>> \Hom_R(\det(J/J^2), L/JL) \\
		@VV\epsilon_{R/J,I}V @VV\epsilon_{\det(J/J^2), L/JL}V \\
		H^d \Hom_{R'}(R'/J', I'^\bullet) @>\cong>> \Hom_{R'}(\det(J'/J'^2), L'/J'L')
		\end{CD}\]
		which is the desired commutative diagram of $R/J$-modules.
		
		Finally, we still need to find a map of complexes $H^d \pi^\flat \I^\bullet[-d] \rightarrow  \pi^\flat \I^\bullet$	in $\Ch^b_c(\Q(Z))$ which is compatible with the involution.\ At this stage, we only have a right roof $ H^d \pi^\flat \I^\bullet [-d] \stackrel{s}\rightarrow C^\bullet \stackrel{t}\leftarrow \pi^\flat \I^\bullet $ compatible with the involution in $\D^b_c(\Q(Z))$, where $C^\bullet$ is the canonical truncation  $\tau_{\leq d}\pi^\flat \I^\bullet$. By this we mean 
		\[C^i = \begin{cases}
		  \pi^\flat \I^d / \textnormal{Im} (\pi^\flat \I^{d-1} \rightarrow  \pi^\flat \I^d ) & \text{if } i =d\\
		  \pi^\flat \I^i & \text{if } d<i \leq n  \\
		  	0 & \text{if otherwise} 
		\end{cases}  \] 
		and the maps of complexes $s$ and $t$ are the obvious canonical quasi-isomorphisms. Since $\pi ^\flat \I^\bullet$ is a complex of injectives, the quasi-isomorphism $t$ is a split injection in the homotopy category  $K^b_c(\Q(Z))$ of complexes, i.e. there is a map of complexes $r: C^\bullet \rightarrow \pi^\flat \I^\bullet$ such that $r \circ t \simeq \id $ in $K^b_c(\Q(Z))$, by  \cite[Lemma 10.4.6]{Wei94}. I claim that any choice of the splitting $r$ will define a quasi-isomorphism of complexes $ rs: H^d \pi^\flat \I^\bullet[-d] \rightarrow  \pi^\flat \I^\bullet$ in $\Ch^b_c(\Q(Z))$ such that the diagram 
		\[ \xymatrix{ H^d \pi^\flat \I^\bullet[-d] \ar[d]_-{ \sigma'_{ \pi^\flat \I^\bullet} := H^d\sigma_{ \pi^\flat \I^\bullet}} \ar[r]^-{rs} &   \pi^\flat \I^\bullet \ar[d]^-{\sigma_{\pi^\flat \I}} \\
		\sigma_{Z*}H^d \pi^\flat \I^\bullet[-d] \ar[r]^-{\sigma_{Z*}(rs)} &  \sigma_{Z*}  \pi^\flat \I^\bullet }\]
	 is commutative in $\Ch^b_c(\Q(Z))$, while at this stage it is only known that it commutes in $\D^b_c(\Q(Z))$. 
	 By \cite[Corollary 10.4.7]{Wei94} and the fact that $\pi^\flat(\I^\bullet)$ is injective, 
 the diagram also commutes in $K^b_c(\Q(Z))$. Let us define $ f: =  \sigma_{\pi^\flat \I}\circ (rs) - \sigma_{Z*}(rs) \circ \sigma'_{\pi^\flat \I^\bullet} $, which is a map of complexes from $H^d \pi^\flat \I^\bullet[-d] $ to $ \sigma_{Z*}  \pi^\flat \I^\bullet$. Then the commutativity in $K^b_c(\Q(Z))$ implies that $f$ is null homotopic. Since the complex $H^d \pi^\flat \I^\bullet[-d] $ is concentrated in degree $d$, we see that the morphism of complexes $f$ is also concentrated in degree $d$. By the definition of null homotopy, there exists a map of $\SO_Z$-modules $u: H^d \pi^\flat \I^\bullet \rightarrow \sigma_{Z*} \pi^\flat \I^{d-1} $ such that $u \circ d^{\pi^\flat \I} =f $. We now want to conclude that $\Hom_{\SO_Z} ( H^d \pi^\flat \I, \pi^\flat \I^{d-1}) \cong \Hom_{\SO_Z} ( H^d \pi^\flat \I, \sigma_{Z*} \pi^\flat \I^{d-1}) $ vanishes. This can be checked locally, and by using that $ht(J) =d$ and $\Hom_{R/J}(M, \Hom_R(R/J, I^{d-1})) = 0$. This last fact is proved in \cite[Lemma 3.3 (5)]{Gil02}. Therefore, one has $u = f = 0$ .   
	\end{proof}
	\begin{Coro}\label{WDevissage}
		Under the hypothesis of Theorem \ref{thm:Devissage},  the direct image map 
		\[ W^{i-d} (Z,\sigma_Z, (\omega_{Z/X} \otimes_{\SO_{X}} \mathcal{L}, \sigma_{\omega_{Z/X}} \otimes  \sigma_{\mathcal{L}_Z }) )  \longrightarrow W^{i}_Z(X,\sigma_X, (\mathcal{L}, \sigma_{\mathcal{L}}) )\] 
		induces an isomorphism of groups.
	\end{Coro}
	
	\section{Some computations}
	Let $S$ be a regular scheme with $\frac{1}{2} \in \SO_S$. Let $\mathbb{P}^1_S$ be the projective line $\mathrm{Proj}(\SO_S[X,Y])$.

	\subsection{$\mathbb{P}^1$ with switching involution} Let $\tau_{\PS^1_S}  : \PS^1_S \rightarrow \PS^1_S : [X: Y] \mapsto [Y:X]$ be the involution which switches the coordinates of $\mathbb{P}^1_S$. 
	\begin{theo}\label{P1switch} Let $S$ be a regular scheme with $\frac{1}{2} \in \SO_S$. Then, we have an equivalence of spectra
		\[ GW^{[i]} (\PS^1_S, \tau_{\PS^1_S}) \cong GW^{[i]} (S) \oplus GW^{[i+1]}(S) \]\ 
		In particular, we have the following isomorphism on Witt groups
		\[W^{i}(\PS^1_S, \tau_{\PS^1_S}) \cong W^{i} (S) \oplus W^{i+1}(S) \]
	\end{theo}
	
	\begin{lem}\label{A1t-t}
		Let $\sigma^{-}_{\mathbb{A}_S^1}: \mathbb{A}_S^1 \rightarrow \mathbb{A}_S^1$ be the involution given by $t \mapsto -t$. Then, the pullback
		\[ p^*: GW^{[i]}(S) \longrightarrow GW^{[i]}(\mathbb{A}_S^1,\sigma^-_{\mathbb{A}_S^1}) \]
		is an isomorphism. 
	\end{lem}
	\begin{proof}
		Let $\pt: S \rightarrow \A^1_S $ be the rational point corresponding to $0$. Then, we have the following commutative diagram
		\[ \begin{tikzcd}[ampersand replacement=\&]
		S \ar[r,"\pt"] \ar[d,"\id"] \& \A^1_S \ar[dl, "p^*"] \\
		S     
		\end{tikzcd}\]
		which induces
		\[ \begin{tikzcd}[ampersand replacement=\&]
		GW^{[i]}(S) \ar[r,"p^*"] \ar[d,"\id"] \&  GW^{[i]}(\mathbb{A}_S^1,\sigma^-_{\mathbb{A}_S^1}) \ar[dl, "\pt^*"] \\
		GW^{[i]}(S)  \&\phantom{a}\hspace{2.2 cm}.   
		\end{tikzcd} \]
		This implies that the pullback $p^*:  GW^{[i]}(S) \longrightarrow  GW^{[i]}(\mathbb{A}_S^1,\sigma^-_{\mathbb{A}_S^1})$ is split injective. Now, consider the composition 
		\[  GW^{[i]}(\mathbb{A}_S^1,\sigma^-_{\mathbb{A}_S^1}) \stackrel{\pt^*}\longrightarrow  GW^{[i]}(S) \stackrel{p^*}\longrightarrow  GW^{[i]}(\mathbb{A}_S^1,\sigma^-_{\mathbb{A}_S^1}). \]
		We want to show that it is an isomorphism. We use the following $C_2$-equivariant homotopy
		\[ \begin{tikzcd}[ampersand replacement=\&]
		\A^1 \times \A^1 \ar[r] \ar[d] \& \A^1 \ar[d]  \&  (t', t) \ar[d, mapsto] \ar[r, mapsto] \& t' t \ar[d, mapsto]  \\
		\A^1 \times \A^1 \ar[r] \& \A^1 \& (t', -t)\ar[r, mapsto] \& -t' t.
		\end{tikzcd}\]
		This diagram tells us that 
		\[(\mathbb{A}_S^1,\sigma^-_{\mathbb{A}_S^1})  \stackrel{p}\longrightarrow (S,\id) \stackrel{\pt}\longrightarrow (\mathbb{A}_S^1,\sigma^-_{\mathbb{A}_S^1}) \] 
		is homotopic to $\id: (\mathbb{A}_S^1,\sigma^-_{\mathbb{A}_S^1})  \rightarrow (\mathbb{A}_S^1,\sigma^-_{\mathbb{A}_S^1}) $. The result follows by Theorem \ref{thm:c2a1invariance} below. \end{proof}
	Let $\mathbb{P}^0_{[1:1]} \hookrightarrow \mathbb{P}^1_S$ be the closed point cut out by the homogeneous ideal $(X-Y)$. Then, $\mathbb{P}^0_{[1:1]}$ is invariant under the involution $\tau_{\PS^1_S}$ on $\mathbb{P}^1_S$.  
	\begin{lem}\label{P1P0support} Let $S$ be a regular scheme with $\frac{1}{2} \in \SO_S$.\ The d\'{e}vissage theorem provides an isomorphism
		\[D_{\mathbb{P}^0/\mathbb{P}^1}:  GW^{[i+1]}(S) \stackrel{\cong}\longrightarrow GW^{[i]}_{\mathbb{P}^0_{[1:1]}} (\PS^1_S, \tau_{\PS^1_S}).   \]
	\end{lem}
	\begin{proof}
		 By Theorem \ref{thm:Devissage}, we conclude that \[ GW^{[i-1]}(\mathbb{P}^0_{[1:1]},\id, (\omega_{\mathbb{P}^0_{[1:1]}/\PS^1_S}, \sigma_{\omega}) ) \stackrel{\cong}\longrightarrow GW^{[i]}_{\mathbb{P}^0_{[1:1]}} (\PS^1_S, \tau_{\PS^1_S}).  \]
		 Let $I = (X-Y)$ be the homogeneous ideal defining $\mathbb{P}^0_{[1:1]} =S$  in $ \mathbb{P}^1_S$. Then, the involution $\tau_{\PS^1_S}: \PS^1_S \rightarrow \PS^1_S$ induces an involution on $I$ 
		  which sends $X-Y$ to $Y-X$, and therefore induces a map $\sigma_{I/I^2}: I/I^2 \rightarrow I/I^2 : X-Y \mapsto Y-X$ of $\SO_S$-module. The map $\mu: I/I^2 \rightarrow \SO_S$ defined by sending $X-Y$ to $1$ is an isomorphism of $\SO_S$-modules. Now, we have a commutative diagram
		 \[ \xymatrix{ I/I^2 \ar[r]^-{\mu} \ar[d]_{\sigma_{I/I^2}} & \SO_S \ar[d]^-{-\id} \\
		  I/I^2 \ar[r]^-{\mu} & \SO_S} \]
		which implies $(\omega_{\mathbb{P}^0_{[1:1]}/\PS^1_S}, \sigma_{\omega}) \cong (\SO_S, -\id)$ and hence 
		\[ GW^{[i-1]}(\mathbb{P}^0_{[1:1]},\id, (\omega_{\mathbb{P}^0_{[1:1]}/\PS^1_S}, \sigma_{\omega}) )  \stackrel{\cong}\longrightarrow GW^{[i-1]}(S,\id, (\SO_S, -\id) )  \stackrel{\cong}\longrightarrow  GW^{[i-1]}(S, - \can). \]
		 Also, $GW^{[i-1]}(S, - \can) \cong GW^{[i+1]}(S)$.
	\end{proof}
	
	\begin{proof}[Proof of Theorem \ref{P1switch}]
		By \cite{Sch17}, we have the following homotopy fibration sequence
		\[\begin{CD}
		GW^{[i]}_{\mathbb{P}^0_{[1:1]}} (\PS^1_S, \tau_{\PS^1_S})  \longrightarrow GW^{[i]} (\PS^1_S, \tau_{\PS^1_S}) \longrightarrow  GW^{[i]} (\PS^1_S - \mathbb{P}^0_{[1:1]} , \tau_{\PS^1_S - \mathbb{P}^0_{[1:1]}}). 
		\end{CD} \]
		Note that there is an invariant isomorphism 
		\[(\mathbb{A}_S^1,\sigma^-_{\mathbb{A}_S^1})  \rightarrow (\PS^1_S - \mathbb{P}^0_{[1:1]} , \tau_{\PS^1_S - \mathbb{P}^0_{[1:1]}}) \] 
		defined by sending $t$ to $[t+\frac{1}{2}: t-\frac{1}{2}]$. Therefore, we get a homotopy fibration sequence
		\[
		\begin{tikzcd}
		GW^{[i]}_{\mathbb{P}^0_{[1:1]}} (\PS^1_S, \tau_{\PS^1_S})\ar[r] & GW^{[i]} (\PS^1_S, \tau_{\PS^1_S})  \ar[r] &  GW^{[i]} (\mathbb{A}_S^1,\sigma^-_{\mathbb{A}_S^1})& \\
		&&GW^{[i]}(S) \ar[u,"\cong" ,"p^*"'] \ar[ul, "q^*"]&\phantom{a}\hspace{-1.8 cm}.
		\end{tikzcd} \]
		By Lemma \ref{A1t-t}, we have the isomorphism $p^*: GW^{[i]}(S) \rightarrow GW^{[i]} (\mathbb{A}_S^1,\sigma^-_{\mathbb{A}_S^1}) $ which provides a splitting for the above homotopy fibration sequence. By combining it with Lemma \ref{P1P0support}, we conclude that
		\[GW^{[i]} (\PS^1_S, \tau_{\PS^1_S})   \cong GW^{[i]}_{\mathbb{P}^0_{[1:1]}} (\PS^1_S, \tau_{\PS^1_S}) \oplus  GW^{[i]}(S) \cong      GW^{[i+1]}(S)\oplus  GW^{[i]}(S)\ .\]
		The result now follows. \end{proof}
	
	\subsection{$C_2$-equivariant $\mathbb{A}^1$-invariance}
	Consider the involution $\sigma_{\mathbb{P}^1_S}:\mathbb{P}^1_S \rightarrow \mathbb{P}^1_S$ given by
	the graded morphism $\SO_S[X,Y] \rightarrow \SO_S[X,Y]$ of graded sheaves of $\SO_S$-algebras, 
	such that $a\mapsto \sigma_S(a)$ if $a \in \SO_S$, and $X \mapsto X, Y\mapsto Y$. There is an element
	\[\beta := 
	\arraycolsep=1.7pt\def\arraystretch{1.7}
	\left( \begin{CD}
	\SO_{\PS^1}(-1) @>X>> \SO_{\PS^1} \\
	@VYVV  @VYVV \\
	\SO_{\PS^1} @>X>> \SO_{\PS^1}(1)  
	\end{CD}\right) \in GW^{[1]}_0(\mathbb{P}^1_S,\sigma_{\mathbb{P}^1_S})\ .  \]
	\begin{Prop}\label{P1} In this result, the base $S$ can be singular, and still $\frac{1}{2} \in \SO_S$. The map of spectra
		\[   (q^*, \beta \cup q^*(-))   : GW^{[i]}(S,\sigma_S) \oplus GW^{[i-1]}(S,\sigma_S) \rightarrow      GW^{[i]} (\mathbb{P}^1_S,\sigma_{\mathbb{P}^1_S}) \]
	 is an equivalence.
	\end{Prop}
	\begin{proof}
		The proof of \cite[Theorem 9.10]{Sch17} can be applied without modification. 
	\end{proof}
	\begin{theo}\label{thm:c2a1invariance} Let $S$ be a regular scheme with involution $\sigma_S$ and with $\frac{1}{2} \in \SO_S$. Let $\sigma_{\A^1_S}: \A^1_S \rightarrow \A^1_S$ be the involution on $\A^1_S$ with the indeterminant fixed by $\sigma_{\A^1_S}$ and such that the following diagram commutes
		\[ \begin{CD}
		\A^1_S @>\sigma_{\A^1_S}>> \A^1_S\\
		@VpVV @VpVV \\
		S        @>\sigma_{S}>> S \ .
		\end{CD}\] 
		Then, the pullback 
		\[ p^*:   GW^{[i]} (S,\sigma_S) \rightarrow GW^{[i]} (\A^1_S , \sigma_{\A^1_S}) \]
		is an isomorphism.
	\end{theo}
	\begin{proof}
		Let $\mathrm{pt} :S \rightarrow  \mathbb{P}^1_S$ be the rational point with $X = 0$ and $Y=1$. It follows that there is a commutative diagram
		\[ \begin{CD}
		S @>\sigma_S>> S \\
		@V{\pt}VV @V{\pt}VV \\
		\mathbb{P}^1_S @>\sigma_{\mathbb{P}^1_S}>> \mathbb{P}^1_S\ .
		\end{CD}\]
		By Schlichting \cite[Theorem 6.6]{Sch17}, we have the following localization sequence
		\[ GW^{[i]}_{\pt} (\mathbb{P}^1_S,\sigma_{\mathbb{P}^1_S}) \longrightarrow GW^{[i]} (\mathbb{P}^1_S,\sigma_{\mathbb{P}^1_S}) \longrightarrow GW^{[i]} (\mathbb{A}^1_S,\sigma_{\mathbb{A}^1_S})\]
		We write out the following commutative diagram
		\begin{equation}\label{P1A1}
		\begin{tikzcd}[ampersand replacement=\&]
		GW^{[i-1]} (S,\sigma_{S}) \ar[r, "{\begin{pmatrix} 0  \\ 1 \end{pmatrix}}"] \ar[dd, "D_{\pt/\mathbb{P}^1_S}"']
		\& GW^{[i]} (S,\sigma_{S}) \oplus GW^{[i-1]} (S,\sigma_{S}) \ar[r, "{\begin{pmatrix} 1 & 0  \end{pmatrix}}"] \ar[dd, "{\begin{pmatrix} q^* & \beta \cup q^*(-)  \end{pmatrix}}"'] 
		\&  GW^{[i]} (S,\sigma_{S}) \ar[dd, "p^*" '] \\ \\
		GW^{[i]}_{\pt} (\mathbb{P}^1_S,\sigma_{\mathbb{P}^1_S})  \ar[r] 
		\& GW^{[i]} (\mathbb{P}^1_S,\sigma_{\mathbb{P}^1_S})  \ar[r] \&  GW^{[i]} (\mathbb{A}^1_S,\sigma_{\mathbb{A}^1_S})\ .
		\end{tikzcd}
		\end{equation}
		The main reason for the commutativity of the left square is that we have the locally free resolution
		\[ 0 \longrightarrow  \SO_{\mathbb{P}^1}(-1) \stackrel{X}\longrightarrow  \SO_{\mathbb{P}^1} \longrightarrow \SO_{\pt} \longrightarrow 0.\]
		
		By d\'{e}vissage, we conclude that  $D_{\pt/\mathbb{P}^1_S}$ is an equivalence. By Proposition \ref{P1}, we know that the middle map ${\begin{pmatrix} q^* & \beta \cup q^*(-)  \end{pmatrix}}$
		 is an equivalence. It follows that the right arrow
		\[ p^* :  GW^{[i]} (S,\sigma_{S})  \rightarrow GW^{[i]} (\mathbb{A}^1_S,\sigma_{\mathbb{A}^1_S}) \]
		 is also an equivalence. 
	\end{proof}
	\subsection{Representability}
	It is easy to see that schemes with involution can be identified with schemes with a $C_2$-action where $C_2$ is the cyclic group of order two considered as an algebraic group. Let $\mathcal{H}^{C_2}_\bullet(S)$ be the $C_2$-equivariant motivic homotopy category of Heller, Krishna and \O stv\ae r \cite{HKO14}.\  Consider the presheaf of simplicial sets $GW^{[i]}: \Sm^\op_S \rightarrow \textnormal{sSet}$ by a big vector bundle argument cf. \cite[Remark 9.2]{Sch17}. The presheaf $GW^{[i]}$ is an object in $\mathcal{H}^{C_2}_\bullet(S)$. We prove that the following representability result.
\begin{theo}\label{thm:C2representablity} Let $(X,\sigma) \in \textnormal{Sm}^{C_2}_S$. Then, there is a bijection of sets
	\[[S^n \wedge (X,\sigma)_+, GW^{[i]}]_{\mathcal{H}^{C_2}_\bullet(S)} = GW_n^{[i]}(X,\sigma). \]
\end{theo}
	\begin{proof}
		By \cite[Corollary 4.9]{HKO14}, we need to prove the $C_2$-equivariant Nisnevich excision and $\A^1$-invariance for Hermitian $K$-theory. Note that the $C_2$-equivariant Nisnevich excision for $GW^{[i]}$ can be proved by a variant of \cite[Theorem 9.6]{Sch17} adapted to schemes with involution since the argument is independent of the duality. Moreover,  $\A^1$-invariance is proved in Theorem \ref{thm:c2a1invariance}.
	\end{proof}
	 \textbf{Acknowledgements.}
	I want to thank Marco Schlichting for useful discussion.\ I appreciate the referee for pointing out a gap in an earlier version.\ I thank Thomas Hudson for the proofreading of the paper. I would like to acknowledge support from the EPSRC Grant EP/M001113/1, the DFG priority programme 1786 and the GRK2240. Part of this work was carried out when I was visiting Hausdorff Research Institute for Mathematics in Bonn and Max-Planck-Institut in Bonn. I would like to express my gratitude for their hospitality.
	
	

\end{document}